\documentclass[12pt]{amsart}
\usepackage{latexsym, amssymb, amscd, amsfonts, floatflt, amsthm}
\input xy
\xyoption{all}

 \newlength{\baseunit}               
 \newcount{\numlines}                
\newcommand{\zed}{\mathbb{Z}}

\newcommand{\bR}{\mathbb{R}}

\newtheorem{theorem}{Theorem}[section]
\newtheorem{cor}[theorem]{Corollary}
\newtheorem{lemma}[theorem]{Lemma}
\newtheorem{prp}[theorem]{Proposition}

\newtheorem{conjecture}[theorem]{Conjecture}
\theoremstyle{remark}
\newtheorem*{rem}{Remark}

\newcommand{\notation}[1]{}
\newcommand{\remind}[1]{{}}
\newcommand{\lremind}[1]{{}}

\newcommand{\secretnote}[1]{}

\begin{document}
\pagestyle{plain}
\title{The resonance method for large character sums}

\author{Bob Hough}
\email{rdhough@math.stanford.edu}
\maketitle

\begin{abstract}
We consider the size of large character sums, proving new lower bounds
for the quantity $\Delta(N,q) = \sup_{\chi\neq \chi_0 \bmod q}
\left|\sum_{n \leq N} \chi(n)\right|$ for almost all ranges of $N$.  Our
results improve those of Granville and
Soundararajan \cite{granville_soundararajan_large_char}, and are are
typically stronger than corresponding bounds known for real character
sums. The results are proven using the resonance method and saddle
point analysis.
\end{abstract}

{\parskip=12pt 

\section{Introduction}
In \cite{granville_soundararajan_large_char} Granville and Soundararajan have
made an extensive study of large character sums.  Varying the
Dirichlet character $\chi$ among non-trivial characters to a large
modulus $q$ they establish  (among other results) new lower bounds for the
quantity
\begin{equation}\label{unitary}\Delta(N,q) = \sup_{\chi \neq \chi_0 \bmod q} \left| \sum_{n \leq
    N}\chi(n)\right|;\end{equation}
varying $D$ among squarefree numbers $q \leq |D| \leq 2q$ they also
consider large values of real character sums, giving lower bounds for
the quantity\footnote{$\left(\frac{D}{\cdot}\right)$ is the Legendre symbol.}
\begin{equation}\label{symplectic} \Delta_\bR(N,q) = \sup_{q \leq |D|
    \leq 2q} \left|\sum_{n \leq N} \left(\frac{D}{n}\right)\right|.
\end{equation}
In both cases their results have been the best known for
essentially all ranges of $N$.  

Omega results of these types are interesting because even
assuming powerful (conjectural) analytic tools like the Generalized Riemann
Hypothesis and bounds derived from Random Matrix Theory for the
associated $L$-functions, the upper bounds that are available for
character sums are not clearly sharp for many ranges of $N$. P\'{o}lya and
Vinogradov proved the classical unconditional bound
\[\left|\sum_{n \leq N}\chi(n)\right| \leq \sqrt{q}\log q\]
and Montgomery and Vaughan \cite{mont_vaughan_grh} improved this to $\ll \sqrt{q}\log \log q$
under assumption of the GRH for
$L(s,\chi)$.  This bound is sharp up to constants, because Paley \cite{paley}
proved that there exist quadratic characters with character sum of
size $\gg \sqrt{q}\log \log q$, and Granville and Soundararajan have
shown the same result for non-quadratic characters in
\cite{granville_soundararajan_pretension}.  Nonetheless, if we believe in roughly square root
cancellation in long character sums, then the GRH bound appears good only if
$N$ is of size
$q^{1-\epsilon}$.  Recently Farmer, Gonek and Hughes \cite{farmer_gonek_hughes} have conjectured that
\begin{equation}\label{L_bound}\left|L(\frac{1}{2} + it; \chi)\right| \ll \exp\left((1 +
  o(1))\sqrt{\frac{1}{2}\log (q+|t|) \log \log (q+|t|)}\right)\end{equation} on
the basis of a calculation involving large unitary random
matrices. This would lead to a bound for smoothed character sums of
\[\sum_n \chi(n)\phi(\frac{n}{N}) \ll_\phi \sqrt{N}\exp\left((1 +
  o(1))\sqrt{\frac{1}{2}\log q \log \log q}\right), \qquad \phi\geq 0
\in C_c^\infty(\bR^+),\] improving the GRH bound for $N<
q^{1-\epsilon}$.  
If the bound (\ref{L_bound}) is true then we might speculate that for $N$
a power of $q$, $N = q^\theta, 0<\theta<1$ that the resulting bound
for $\sum \chi(n)\phi(\frac{n}{N})$ is nearly sharp; one reason for
this belief is that in our Theorem \ref{large_theorem}, below, we demonstrate that for
any such $q, N$ there are many non-principal $\chi$ modulo $q$ for which
\[\left|\sum_{n \leq N} \chi(n)\right| \geq \sqrt{N}\exp\left((1 +
  o(1))\sqrt{(1-\theta) \frac{\log q}{\log_2 q}}\right).\]

When $N < \exp\left(\sqrt{2 \log q \log \log
    q}\right)$ even the random matrix bound is trivial.  One insight
into large character sums for such small $N$ is available through a
conjecture of Granville and Soundararajan. For arithmetic function $f$, write as in
\cite{granville_soundararajan_large_char} 
\[\Psi(x,y;f) = \sum_{\substack{n \leq x\\ p|n \Rightarrow p \leq y}}
f(n)\] for the summatory function of $f$ restricted to numbers having
their largest prime factor at most $y$, and also set
$\Psi(x,y):=\Psi(x,y;1)$ for the number of such '$y$-smooth numbers' less
than $x$.
Granville and Soundararajan have proposed the following conjecture.
\begin{conjecture}[\cite{granville_soundararajan_large_char}
  Conjecture 1]\label{sound_gran_conj} There exists a constant $A > 0$ such that
\[\sum_{n \leq N}\chi(n) = \Psi(N,y;\chi) + o(\Psi(N,y;\chi_0))\]
for $y = (\log q + \log^2 N)\log_2^A q.$
\end{conjecture}
\noindent In particular, this conjecture  implies the upper bound \begin{equation}\label{conjecture_upper_bound}\left|\sum_{n \leq
    N}\chi(n)\right| \ll \Psi(N, (\log q + \log^2 N)\log_2^A q).\end{equation}
Note that in the range $\log N > \sqrt{\log q}$ where the
$\log^2 N$ term dominates, $\Psi(N, \log^2 N)$ is  of size
$\sqrt{N} \exp(\frac{\log N}{\log_2 N})$, which is already much larger than
the random matrix bound if $N$ is a small power of $q$;  there is not
wide-spread consensus, however, regarding the
conjectured bound (\ref{L_bound}), and so it is plausible that even this larger
bound for
character sums is near the truth.  When $\log N < \sqrt{\log q}$, by modifying
an argument of Granville
and Soundararajan (their Theorem 2) one can show\footnote{Although we do not
 do so here.} that the upper bound
(\ref{conjecture_upper_bound}) with, say, $A = 3$ follows from (\ref{L_bound}). 
Conversely, in our Theorem \ref{small_prime_theorem} we establish that if $q$ is
prime and $\log N = (\log q)^{\frac{1}{2} - \epsilon}$ then there is
non-principal character $\chi$ modulo $q$ with \[\left|\sum_{n \leq
    N}\chi(n)\right| \geq \Psi(N, \log q \log_2^{1 - o(1)}q);\] thus if
Conjecture
\ref{sound_gran_conj} is to hold, one must take $A \geq 1$.

There is
some heuristic reason to think that the constant $A = 1$ may be
the right one: given roughly $q$ characters modulo $q$, we might
expect that the first $\log q$ values of $\chi(p)$ may be 
correlated towards 1 for some character $\chi$, which would produce a
positive bias on the $\log q \log_2 q$-smooth numbers.   We should point out,
however, that our argument in Theorem \ref{small_prime_theorem} does not
condition the first
$\log q$ primes, but rather finds a smaller bias among many primes
that are larger than $\log q \log_2 q$.  Thus there may be reason to
believe that the large values of $\sum_{n \leq N}\chi(n)$ are even
larger than $\Psi(N, \log q \log_2 q)$.

\subsection{Discussion of previous work}
Surprisingly, the methods used in \cite{granville_soundararajan_large_char} in
treating general characters and  real characters
are  not at all related; the lower bounds  for
$\Delta(N,q)$ are produced by taking high moments of the character
sums\footnote{Essentially the $2k$th moment where $k = \lfloor
  \frac{\log q}{\log N} \rfloor$.}, while for $\Delta_\bR(N,q)$ the argument appeals to quadratic
reciprocity and Dirichlet's theorem on primes in arithmetic
progressions to first restrict to prime $|D|$ for which
$\left(\frac{D}{p}\right) = 1$ for all small  $p \ll \log q$.
This produces a large positive bias in the sum coming from the
smooth numbers that have all of their prime factors less than $\log
q$.  

When $N$ is relatively small compared to $q$, in the range $\log N <
\sqrt{\log q}$, the two methods perform roughly equally although the
results for real characters are slightly stronger.  
For larger $N$ this
difference becomes more pronounced, essentially because the first
method is most effective when the moment taken is quite large.  Note,
however, that the two methods are not directly comparable in
\cite{granville_soundararajan_large_char} because the results for large real
character sums are produced for conductors $D$ that are prime, whereas
the results for general characters
are stated primarily for the worst case when $q$ is the product of
many distinct small primes; when the conductor $q$ of a character
sum is highly composite in this way it greatly reduces the size of the
large character sums.  To give a trivial example, it may transpire
that $\Delta(q, \log q)
= 0$ since it is possible that for all $n \leq \log q$, $(n,q) > 1$.  

In this paper we adapt the 'resonance method' introduced by
Soundararajan in \cite{sound_resonance} to prove new bounds for
$\Delta(N,q)$.  We consider separately the bounds that this obtains
for $q$ prime and for any $q$.  For general $q$ we improve the bounds
for $\Delta(N,q)$ in \cite{granville_soundararajan_large_char} for all
$N$ larger than a fixed power of $\log q$. For prime $q$ our bounds
are stronger than those obtained in
\cite{granville_soundararajan_large_char} for $\Delta_\bR(N,q)$ for
all $N$ in the range $\exp((\log \log q)^2) \ll N \ll q \exp(-(\log
\log q)^2)$; outside this range the bounds given by
\cite{granville_soundararajan_large_char} were already (at least
conjecturally) best possible.  Moreover, for large N,
$\exp(\sqrt{\log q}) < N < q^{1-\epsilon}$ our improvements over the
previous bounds are substantial.  

We will describe our results in greater detail in the next
section, but we first pause to explain the resonance method in brief, 
 and its conceptual advantage  over the two
methods previously developed in
\cite{granville_soundararajan_large_char}; there is a sense in which
this method generalizes each of the earlier approaches.
The starting point is the simple inequality
\begin{equation}\label{basic_inequal}\min\{x_1, ..., x_n\} \leq \frac{w_1 x_1  + ... + w_n
  x_n}{w_1 + ... + w_n} \leq \max\{x_1, ..., x_n\},\end{equation} which is valid
for any non-negative weights $w_1, ..., w_n$ with some $w_i \neq 0$.
In the resonance method for character sums, the indices are
characters, the variables $x_\chi = \sum_{n \leq N}\chi(n)$ or $x_\chi
= \left|\sum_{n \leq N} \chi(n)\right|^2$ are
character sums, and the weights
are  (squared norms of)  Dirichlet polynomials:
\[w_\chi = \left|\sum_{n \leq x} r(n)\chi(n) \right|^2, \qquad x \leq \frac{q}{N}.\]  Here the
coefficients $r(n)$ are fixed, non-negative, and multiplicative, and are
chosen so as to
maximize the ratio in (\ref{basic_inequal}). When
$N$ is small, a good (but not optimal) choice for the weight $w_\chi$
is
\[w_\chi = \left|\left(\sum_{n \leq N}\chi(n)\right)^{k-1}\right|^2,\] and with
this choice of weights, the resonance method is seen to contain the
first method of \cite{granville_soundararajan_large_char}.  On the
other hand, we are free to choose a weight of the form
\[w_\chi = \left|\prod_{p < P} (1 + \chi(p))\right|^2,\] which has the
effect of placing much more emphasis on those $\chi$ for which
$\chi(p) \approx 1$ for many small primes.  Thus the resonance method
can be interpretted as taking a conditional expectation of character
sums with optimal conditioning, which, at least philosophically,
extends the second method of
\cite{granville_soundararajan_large_char}.

\subsection{Precise statement of results}
Our lower bounds for $\Delta(N,q)$ come in three forms: we give lower
bounds for $\Delta(N,q)$ that hold when $q$ is prime, and for any
$q$.  When $q$ is prime we also consider the dual problem of 'long'
character sums, and give lower bounds for $\Delta(\frac{q}{N}, q)$.  Recall that the Erd\"{o}s-Kac Theorem says that a 'typical'
number of size $x$ has $\sim \log \log x$ distinct prime factors; our
lower bounds for $\Delta(N,q)$, $q$ prime in fact apply equally well
if $q$ is typical, and even if $q$ has $\log^{1-\epsilon} q$ prime
factors, but we restrict to the case that $q$ is prime to ease the
exposition.    
For the dual problem we rely on a 'Fourier expansion' of character
sums, due to P\'{o}lya, that is only valid for primitive $\chi$; in
this case the  restriction to prime $q$ seems to be necessary to the
method.

In our first two theorems we consider the range $\log N < \sqrt{\log
  q}$.
\begin{theorem}\label{small_prime_theorem}
Let $q$ be a large prime and let $\log N < \sqrt{\log q}$ and define functions
\[G(\sigma) = \frac{\Gamma(1 -
    \frac{1}{2\sigma}) \Gamma(\frac{2}{\sigma} -
    1)}{\Gamma(\frac{3}{2\sigma})}, \qquad \qquad \kappa(\sigma) =
  2^{\frac{-1}{1-\sigma}} G(\sigma)^{\frac{\sigma}{1-\sigma}}.\] 
Let $\sigma > \frac{1}{2}$
  solve
$ \log N = \frac{(\log q)^{1-\sigma}}{2(1-\sigma)}G(\sigma)^\sigma$.  Then
\[ \Delta(N,q) \geq \Psi(N, (1 + o(1))\kappa(\sigma) \log q) .\]  
Let $\sigma' > \frac{1}{2}$ solve $\log N = \frac{(\log
q)^{1-\sigma'}}{2^{2-\sigma'}(1-\sigma')}G(\sigma')^{\sigma'}$.  Then
\[\Delta(\frac{q}{N},q) \gg \frac{\sqrt{q}}{N}\Psi(N, (\frac{1}{2} +
o(1))\kappa(\sigma')\log q).\]  Furthermore, if $\log N < \log_2^2 q
\log_3^{-10}q$ then
\[\Delta(N,q) \gg \sqrt{q}\frac{\log \log
  q}{\log\left(\frac{\log N}{\log \log q}\right)} \frac{\Psi(N,\log q)}{N}.\]
\end{theorem}
\noindent
The function $\kappa(\sigma)$ is decreasing on $(\frac{1}{2},1)$.  It satisfies
$\lim_{\sigma \to 1} \kappa(\sigma) = \frac{8}{e^3}$ and $\kappa(\sigma) \sim
\frac{1}{4\sigma - 8}$ as $\sigma \downarrow \frac{1}{2}$.  

We can
deduce the following corollary.
\begin{cor}
Suppose $\log N = (\log q)^{o(1)}$.  Then \[\Delta(N,q) \geq \Psi(N,
(\frac{8}{e^3} + o(1))\log q), \qquad \qquad \frac{8}{e^3} = 0.3982965...\]  
If $\frac{1}{2} < \sigma < 1$ is fixed and $\log N = (\log
q)^{1-\sigma}$ then \[\Delta(N,q) \geq \Psi(N, (\kappa(\sigma) +
o(1))\log q).\]
Finally, if instead $\log N = \sqrt{\log q}\exp(-(\log_2 q)^{o(1)})$ then \[\Delta(N,q)
\geq \Psi(N, \log q \log_2^{1-o(1)}q).\]
\end{cor}
For $q$ prime, in \cite{granville_soundararajan_large_char} Theorem 3 the
bound $\Delta(N,q) \geq \Psi(N, \log q)$ was proved for $\log N <
\frac{\log_2^2 q}{\log_3^2 q}$.  In this range we have $\Psi(N, \log q)
\sim \Psi(N, \frac{8}{e^3}\log q)$ and so our theorem extends this result to 
$\log N < \sqrt{\log q}$.  A more
direct comparison is to \cite{granville_soundararajan_large_char}
Theorem 9, where in the range $\log N < \sqrt{\log q}$ they prove give a bound
for real characters of
\[\Delta_\bR(N,q) \geq \Psi(N, \frac{1}{3} \log q).\]  In the range
$\log N = (\log q)^{o(1)}$ the bound from Theorem
\ref{small_prime_theorem} is already superior, and as $\log N$
increases, the ratio $\frac{\kappa(\sigma)\log q}{\frac{1}{3} \log q}$
in the number of smooth numbers taken,
tends to 
$\infty$.  Regarding the dual statement,
\cite{granville_soundararajan_large_char} Theorem 8 previously gave
the bound
\[\max_{t \leq \frac{q}{N}} \max_{\chi \neq \chi_0 \bmod q}
\left|\sum_{n \leq t} \chi(n)\right| \gg \sqrt{q}\frac{1}{N} \Psi\left(N,
\frac{\log q}{(\log_2 q)^{10}}\right);\] our theorem thus removes the
need for a second maximum over $t$, and improves the bound.  The gain
over the previous result for real characters
(\cite{granville_soundararajan_large_char} Theorem 11) is comparable
to the improvement for $\Delta(N,q)$.  

We next state our result for general moduli $q$.
\begin{theorem}\label{small_composite_theorem}
Let $q$ be any large integer and let $N$ be such that $\log^B q < N <
\exp(\sqrt{\log q})$ for a sufficiently large fixed constant $B$.
In the range $\log N < \log_2^3 q \log_3 q$, there exists a parameter
$\eta = \eta(N,q) = (1 + o(1))\log_2 q$ such that,
\[\Delta(N,q) \geq \frac{N^{\frac{1}{2} + \frac{\lfloor
      u'\rfloor}{2u'}}}{(\log N)^{\lfloor u' \rfloor}}
\left(\frac{e}{\sqrt{2} + o(1)}\right)^{\lfloor u' \rfloor}, \qquad
\qquad u' = \frac{\eta}{\eta + 1} \frac{\log N}{\log_2 q}.\]
In the wider range $\log N \gg \log_2^3 q \log_3 q$, set $\log N =
(\log q)^{1-\sigma}$.  We have
\[\Delta(N,q) \geq \frac{N}{(\log N)^u} \left(\frac{e -
    o(1)}{\sqrt{2\sigma(2\sigma-1)}}\right)^u, \qquad \qquad u =
\frac{\log N}{\log_2 q}.\]
\end{theorem}
\noindent
This theorem should be compared with
\cite{granville_soundararajan_large_char} Theorem 4; for all $N$
larger than a fixed power of $\log q$ we obtain an improvement which
is at least exponential in $u$; as $\log N$ increases to $\sqrt{\log
  q}$ the improvement is larger than any fixed exponential in $u$.

Next we consider the range $\log \log N = (\frac{1}{2} + o(1))\log
\log q$.
\begin{theorem}\label{moderate_theorem}
Let $q$ be any large integer and let $\log N = \tau \sqrt{\log q
  \log_2 q}$ with $\tau = (\log_2 q)^{O(1)}.$  Define $A$ and $\tau'$
by solving
\[\tau = \int_{A}^\infty e^{-x}\frac{dx}{x}, \qquad\qquad \tau' = \int_{A}^\infty e^{-x} \frac{dx}{x^2}.\]  Then
\[\Delta(N,q) \geq \sqrt{N} \exp\left((1 + o(1))A(\tau +
  \tau')\sqrt{\frac{\log q}{\log_2 q}}\right).\]  If $q$ is prime then
instead define $A$ and $\tau'$ by 
\[\sqrt{2}\tau = \int_A^\infty e^{-x}\frac{dx}{x}, \qquad \qquad \tau =
\int_{A}^\infty e^{-x}\frac{dx}{x^2}.
\]
With this new definition we have
\[\Delta(\frac{q}{N}, q) \geq \sqrt{\frac{q}{N}} \exp\left((1 +
  o(1))A\left(\tau + \frac{\tau'}{\sqrt{2}}\right)\sqrt{\frac{\log q}{\log_2
q}}\right).\]
\end{theorem}
\noindent Note that as $\tau \to \infty$, $A\tau \to 0$ and $A \tau' \to 1$.
Meanwhile, as $\tau \to 0$, $A \to \infty$ and $\tau \sim
\frac{e^{-A}}{A}$, $\tau' \sim \frac{e^{-A}}{A^2}.$  This theorem
should be compared to \cite{granville_soundararajan_large_char}
Theorems 5 and 8; a direct comparison is difficult because their
statement is not explicit, but asymptotically our result is stronger
as $\tau \to 0$ and $\tau \to \infty$.

Finally we consider longer character sums.
\begin{theorem}\label{large_theorem}
Let $q$ be a large integer and let $4\sqrt{\log q
\log_2q} \log_3 q< \log N$ and $N = q^\theta$ with $\theta <
1-\epsilon$.  Then 
\[\Delta(N,q) \geq \sqrt{N}\exp\left((1 + o(1)) \sqrt{\frac{(1-\theta)
      \log q}{\log_2 q}}\right).\]
If in addition $q$ is prime then
\[\Delta(\frac{q}{N}, q) \geq \sqrt{\frac{q}{N}} \exp\left((1 +
  o(1))\sqrt{\frac{(1-\theta)\log q}{2\log_2 q}}\right).\]
\end{theorem}
\noindent These bounds are substantially larger than those proved in
\cite{granville_soundararajan_large_char} Theorems 6, 7, and 8 for
$\Delta(N,q)$ in the corresponding range.  The improvement is most noticeable
when $N = q^\theta$ is a power of $q$; for this range
\cite{granville_soundararajan_large_char} Theorem 7 gives only $\Delta(N,q) >
\sqrt{N}(\log q)^{O(\theta^{-1})}$. Our bound is also larger
than the one for real characters in
\cite{granville_soundararajan_large_char} Theorem 10:
\[\Delta_{\bR}(N,q) \geq \sqrt{N} \exp\left((1 + o(1))\frac{\sqrt{\log
      q}}{\log_2 q}\right).\]

\section{The basic propositions and outline of proofs}
The basic proposition of the resonance method is the following.
\begin{prp}[Fundamental Proposition]\label{fundamental_prop}
 Let $N< \sqrt{q}$ and  $x = \frac{\phi(q)}{N}$.  Let $r(n)$ be a completely
multiplicative function satisfying $r(p) \geq 0$ and $p|q \Rightarrow r(p) =
0$. Set 
\begin{equation}\label{squares_condition} B = \frac{\sum_{n = 1}^\infty
r(n)^2}{\sum_{n \leq \frac{x}{N}} r(n)^2}.\end{equation}
Then
\[\Delta(N,q) \geq O(1) + \frac{1}{B}
\sum_{n \leq N} r(n) .\] Furthermore, let $M$ be minimal such that $\sum_{p \leq
M} \log p > \log q$; the bound remains valid if the restriction $p | q
\Rightarrow r(p) = 0$ is replaced with $p\leq M \Rightarrow r(p) = 0$.
\end{prp}

\begin{rem} The following proof will show (essentially) that
\[\Delta(N, q) \geq \sum_{n \leq N} r(n) \left[\frac{\sum_{m \leq \frac{x}{n}}
r(m)^2}{\sum_{m \leq x}r(m)^2} \right]\] 
for all non-negative multiplicative functions $r$.  
In practice we will apply
Proposition \ref{fundamental_prop} with $B = 1+o(1)$ as $q \to \infty$ so that
we aim to solve the optimization problem
\begin{align}
\label{optimization} \text{Maximize:} & \qquad \sum_{n \leq N} r(n)\\
\label{constraint} \text{Subject to:} & \qquad \sum_{n \leq \frac{x}{N}} r(n)^2
= (1 + o(1))
\sum_n r(n)^2 
\end{align}
Here the constraint condition (\ref{constraint}) is closely related to the
condition at primes
\[ \sum_p \log p \frac{r(p)^2}{1-r(p)^2} <
\log x\]
via the saddle point method.  If we assume that the optimal choice for $r$ in
(\ref{optimization}) satisfies $\sum_{n \leq N}r(n) = N^\theta$, then 
$\sum_{n \leq \frac{N}{p}} r(n) \sim \left(\frac{N}{p}\right)^\theta$.  Thus
maximization of (\ref{optimization}) with respect to (\ref{constraint}) via
Lagrange multipliers leads to the heuristic solution
\[ \frac{r(p)}{(1-r(p)^2)^2} \doteq \lambda \frac{\log p}{p^\theta},\] which
guides our choice of resonator functions. 

One might reasonably wonder whether the imposed condition $B = 1 + o(1)$ is
superfluous; we could
instead suppose that $\sum_{n \leq x} r(n)^2 \sim
x^\alpha$. Then $ \sum_{n \leq
\frac{x}{m}} r(n)^2 \doteq 
\left(\frac{x}{n}\right)^\alpha$ so that we would instead obtain the
optimization problem
\begin{align*}
 \text{Maximize:} & \qquad \sum_{n \leq N} r(n)n^{-\alpha}\\
 \text{Subject to:} & \qquad \sum_{n \leq x} r(n)^2 \sim x^\alpha.
\end{align*}
Replacing $r$ with $\tilde{r}(n) = \frac{r(n)}{n^\alpha}$, this is is
subsumed in the previous optimization problem.
\end{rem}

\begin{proof}
 Define 'resonator' $R(\chi) = \frac{1}{\sqrt{\phi(q)}}\sum_{n\leq x} r(n)
\chi(n)$. Plainly 
\[\Delta(N, q) \geq \sum_{\chi \neq \chi_0}|R(\chi)|^2 \sum_{n \leq N} \chi(n)
\Bigg / \sum_{\chi} |R(\chi)|^2.\]  By orthogonality of characters, the
denominator is $\sum_{n \leq x} r(n)^2$.  Meanwhile the numerator is
\begin{align*}\sum_\chi |R(\chi)|^2 \sum_{n \leq N} \chi(N) - |R(\chi_0)|^2
\sum_{n \leq N} \chi_0(n) &= \sum_{n \leq N} \sum_{m \leq \frac{x}{n}} r(m)
r(mn) - O\left(\frac{xN}{\phi(q)} \sum_{n \leq x}r(n)^2\right) \\ & \geq
\left( \frac{1}{B} \sum_{n \leq N} r(n) - O(1) \right) \sum_{n \leq x} r(n)^2,
\end{align*}
which proves the first statement in the proposition.  

To prove the second
statement, let $r$ be any non-negative
completely multiplicative function supported on primes larger than $M$, and set
$B = \frac{\sum_{n}r(n)^2}{\sum_{n \leq \frac{x}{N}}r(n)^2}.$  Enumerate
\[\{p>M: p|q\} = \{q_1, ..., q_R\}, \qquad \qquad \{p \leq M: p \nmid q\} =
\{p_1, ..., p_S\}.\] Note that $S \geq R$. Now   swapping the values 
$r(p_1),..., r(p_R)$ with $r(q_1), ..., r(q_R)$ we define a new completely
multiplicative function $\tilde{r}$. Obviously $\tilde{r}$ is supported on
primes not dividing $q$, and also
\begin{align*}
 \sum_n \tilde{r}(n)^2  = \sum_n r(n)^2, &\qquad \sum_{n \leq \frac{x}{n}}
\tilde{r}(n)^2  \geq \sum_{n \leq \frac{x}{n}}r(n)^2 \qquad \Rightarrow \qquad
\frac{\sum_n \tilde{r}(n)^2}{\sum_{n \leq \frac{x}{n}}r(n)^2} \leq B\\
 \sum_{n \leq N} \tilde{r}(n) & \geq \sum_{n \leq N}r(n)
\end{align*}
The first part of the proposition applied to $\tilde{r}$ thus gives
\[\Delta(N,q) \geq O(1) + \frac{1}{B} \sum_{n \leq N} r(n).\]

\end{proof}

For primitive characters $\chi$ P\'{o}lya proved the Fourier
expansion\footnote{We use the notation $e(x) = e^{2\pi i x}$ and $c(x) =
\cos(2\pi x)$.} (see \cite{mont_vaughan_grh} Lemma 1)
\begin{equation}\label{polya_fourier}
 \sum_{n \leq \frac{q}{N}} \chi(n) = \frac{\tau(\chi)}{2\pi i}
\sum_{\substack{|h| \leq H\\ h\neq 0}} \frac{\overline{\chi}(h)}{h}(1 -
e(-\frac{h}{N})) + O(1 + qH^{-1}\log q).
\end{equation}
Choose $H = \sqrt{Nq} \log q$ and define $S(\chi) = \sum_{\substack{|h|\leq H\\
h \neq 0}} \frac{\overline{\chi}(h)}{h}(1 - c(\frac{h}{N}))$.  If
$\chi$ is even
($\chi(1) = 1$)  this vanishes, but for odd primitive characters $\chi$ we
 have 
\begin{equation}\label{S_of_x}\left|\sum_{n \leq N} \chi(n)\right| =
\frac{\sqrt{q}}{2\pi} |S(\chi)| +
O(\sqrt{\frac{q}{N}}).\end{equation}  Using this relation we now prove a dual
version of our
main proposition.

\begin{prp}[Fundamental Proposition, dual version]\label{fund_dual}
 Let $q$ prime, $N < \sqrt{q}$, $H = \sqrt{qN}\log q$, $x =
\frac{q}{2H}$ and $N'<H$ a parameter. Let $r$ be a
non-negative completely multiplicative function and put 
\[B_{N'} = \frac{\sum_n r(n)^2}{\sum_{n \leq \frac{x}{N'}}r(n)^2}.\] We have
\[\Delta(\frac{q}{N}, q) \geq \frac{\sqrt{q}}{\pi} \frac{1}{B_{N'}}\sum_{0 <
n
\leq N'} \frac{r(n)}{n}(1-c(\frac{n}{N})) + O(\sqrt{\frac{q}{N}}).\]
Specializing to $N' =
N/2$ we obtain
\[\Delta(\frac{q}{N},q) \gg \frac{\sqrt{q}}{N^2} \frac{1}{B_{N/2}} \sum_{n
\leq N}n r(n) + O(\sqrt{\frac{q}{N}}).\]
\end{prp}
\begin{proof}
 Define, as before, `resonator' $R(\chi) = \frac{1}{\sqrt{q-1}}\sum_{n \leq x}
r(n) \chi(n)$. Then \[\Delta(N,q) + O(\sqrt{\frac{q}{N}}) \geq
\frac{\sqrt{q}}{\pi}\sup_\chi |S(\chi)| \geq \frac{\sqrt{q}}{\pi} \sum_{\chi}
|R(\chi)|^2 S(\chi) \Bigg / \sum_\chi |R(\chi)|^2.\]  The denominator is
$\sum_{n \leq x}r(n)^2$.  Since $x = \frac{q}{2H}$, the sum in the numerator is
\[\sum_{0 \neq |h| \leq H}\frac{(1 - c(\frac{h}{N}))}{h}\sum_{\substack{n_1, n_2
\leq x\\ n_1 h \equiv n_2 \bmod{q}}} r(n_1)r(n_2) = \sum_{0 < h \leq H}
\frac{r(h)(1-c(\frac{h}{N}))}{h} \sum_{n \leq \frac{x}{h}}r(n)^2.\]  Since all
terms in the numerator are positive, we can discard those $h > M$ to obtain the
result.
\end{proof}

\subsection{Outline of proofs, and lemmas}
After introduction of a 'resonator' multiplicative function $r(n)$, the
proofs of Theorems \ref{small_prime_theorem}-\ref{large_theorem} proceed in
the same three steps.

\begin{enumerate}
 \item [A.] Let $x$ be the length of the resonator and $N$ the length of the
sum.  We check that \[\sum_{n \leq \frac{x}{N}}r(n)^2 =
(1 + o(1)) \prod_p (1-r(p)^2)^{-1}\]
so that $B$ in Proposition \ref{fundamental_prop} may be
taken as $1 + o(1)$. 

 \item [B.] From the Fundamental Proposition and part A it follows  $\Delta(N,q)
\sim
\sum_{n \leq N} r(n)$.  We determine the asymptotic
shape of this sum via the Perron integral
\[\sum_{n \leq N} r(n) = \frac{1}{2\pi i} \int_{\sigma-i\infty}^{\sigma + i
\infty} R(s) N^s \frac{ds}{s}, \qquad \qquad R(s) = \prod_p (1 -
r(p)p^{-s})^{-1}\]
 and a saddle point calculation.
 \item [C.] We analyze the various implicitly defined parameters that arise in
the saddle point calculation of part B in order to obtain explicit lower bounds
for $\Delta(N,q)$.
\end{enumerate}

The bound in the first step (A) is accomplished by an appeal to the following
simple lemma.

\begin{lemma}\label{tail_bound_lemma}
 Let $f_i(n)$ be a sequence of non-negative, completely multiplicative
functions satisfying $f_i(n) < 1$, and let $y_i \to \infty$ be a growing
sequence of parameters. Define $\alpha_i = \frac{(\log \log y_i)^2}{\log
y_i}$.  Suppose
\[ \sum_p  \log p \frac{f_i(p)}{1-f_i(p)} < \log y_i - \frac{\log y_i}{\log
\log y_i}\]
and
\[\sum_{k \log p > \frac{\log y_i}{(\log \log y_i)^4}} f(p_i)^k p_i^{k\alpha} =
o(1), \qquad \qquad i \to \infty.\]
Then 
\[\sum_{n \leq y_i} f_i(n) = (1 + o(1)) \sum_{n = 1}^\infty f_i(n),
\qquad\qquad i \to \infty.\]
\end{lemma}

\begin{proof}
 By 'Rankin's trick,'
\[\sum_{n \leq y_i} f_i(n) = \sum_{n = 1}^\infty f_i(n) + O\left(y_i^{-\alpha_i}
\sum_{n = 1}^\infty f_i(n)n^{\alpha_i}\right).\] The logarithm of the ratio
between the error term and the main term is
\begin{equation}\label{log_ratio}-\alpha_i \log y_i + \log \prod_p \left(\frac{1
- f_i(p)}{1-f_i(p)p^{ \alpha_i}}\right).\end{equation}  The logarithm of the
infinite product is 
\begin{align*} \sum_{p, k} \frac{f_i(p)^k}{k} \left[p^{k \alpha_i} - 1\right]
&\leq \left(1 + O(\log\log^{-2}y_i)\right) \alpha_i \sum_{p} \log p
\frac{f_i(p)}{1-f_i(p)} \\ &+ \sum_{k \log p > \frac{\log y_i}{(\log\log
y_i)^4}} \frac{f_i(p)p^{k\alpha_i}}{k}\\ & \leq \alpha_i \left(\log y_i -
\frac{\log y_i}{\log \log y_i} + O(\frac{\log y_i}{(\log \log y_i)^2})\right) +
o(1);
\end{align*}
which proves that the quantity in (\ref{log_ratio}) tends to $-\infty$ as $i
\to \infty$.
\end{proof}

The analysis in the second step (B) closely follows the corresponding
analysis of smooth numbers contained in \cite{hildebrand_tenenbaum}.  We
briefly recall this theory and quote the results from it that we will need.

Recall that we set $\Psi(x,y) = \#\{n \leq x: p|n \Rightarrow p \leq y\}$ for the number of
$y$-smooth numbers less than $x$ and  set also
\[\zeta(s,y) = \prod_{p < y} (1-p^{-s})^{-1}\] for the corresponding generating
function.  Analysis of $\Psi(x,y)$ depends on the behavior of the logarithm of
$\zeta(s,y)$ and its first few derivatives,
\[\psi(s,y) = \log \zeta(s,y), \qquad \psi_j(s,y) = (-1)^j \frac{d^j}{ds^j}
\psi(s,y), \; j = 0, 1, 2, ...\] 
for $s$ near the saddle point $\alpha = \alpha(x,y)>0,$ solving
$\psi_1(\alpha, y) = \log x.$  The basic result is the following.

\begin{theorem}\label{smooth_saddle}
 Uniformly in the range $x \geq y \geq 2$, we have
\[\Psi(x,y) = \frac{x^\alpha \zeta(\alpha,y)}{\alpha \sqrt{2\pi
\psi_2(\alpha,y)}}\left\{1 + O\left(\frac{\log y}{\log x} + \frac{\log
y}{y}\right)\right\}.\]
\end{theorem}
\begin{proof}
 This is \cite{hildebrand_tenenbaum} Theorem 1.6.
\end{proof}

The first stage in the proof of Theorem \ref{smooth_saddle} is to write
\[\Psi(x,y) = \frac{1}{2\pi i} \int_{\alpha - i \infty}^{\alpha + i \infty}
\zeta(s,y) x^s \frac{ds}{s}\] and to truncate the integral at some height $T$. 
The following lemma, which we use, is essentially the one given in
\cite{hildebrand_tenenbaum} to bound the error from truncation.
\begin{lemma}\label{truncation}
 Let $f(n)$ be a bounded, non-negative arithmetic function with Dirichlet
series $F(s) = \sum_{n = 1}^\infty \frac{f(n)}{n^s}$, which converges
absolutely at $s = \sigma_0>0$.  Uniformly in $x\geq 2$, $T \geq 1$,
$\sigma \geq \sigma_0$ and $0<\tau < T$ we have
\[\sum_{n \leq x} f(n) = \frac{1}{2\pi i}\int_{\sigma - i\tau}^{\sigma +
i\tau} F(s)
\frac{x^s}{s} ds + R\]
where $R$ is bounded by 
\[ R \ll 1+  x^{\sigma}F(\sigma)\left(T^{\frac{-1}{2}} +
 \sup_{\tau \leq t \leq T} \frac{|F(\sigma + it)|}{F(\sigma)}
\log T\right).\]
\end{lemma}
\begin{proof}
 See \cite{hildebrand_tenenbaum} Lemma 3.3.
\end{proof}

For the purpose of making comparisons in Theorem \ref{small_prime_theorem} it will be sufficient for
us to know the asymptotic behavior of $\log \Psi(x,y)$.  In this case, the
behavior is well understood in a wide range of $x$ and $y$.  Set $u = \frac{\log
x}{\log y}$ and let $\rho(u)$ denote the Dickman-de Bruijn function. 

\begin{theorem}
 For any fixed $\epsilon > 0$ we have
\[\log(\Psi(x,y)/x) = \left\{1 + O(\exp\{-(\log
u)^{3/5-\epsilon}\})\right\}\log \rho(u)\]
uniformly in the range
\[ y \geq 2, \qquad 1\leq u \leq y^{1-\epsilon}.\]
\end{theorem}

\begin{proof}
 This is \cite{hildebrand_tenenbaum} Theorem 1.2.
\end{proof}

\begin{theorem}
 For $u \geq 1$ we have
\[\rho(u) = \exp\left\{-u\left(\log u + \log_2(u+2) -1 +
O\left(\frac{\log_2(u+2)}{\log u}\right)\right)\right\}.\]
\end{theorem}
\begin{proof}
 This is \cite{hildebrand_tenenbaum} Corollary 2.3.
\end{proof}

In particular we have the following simple lemma.

\begin{lemma}\label{log_Psi_change}
 Let $u = \frac{\log x}{\log y}$ as above.  When $u < \sqrt{y}$ and for
$|\kappa| < 1$ we have
\begin{equation}\label{kappa_smooth}\log \frac{\Psi(x, e^\kappa y)}{\Psi(x,y)} =
\left(\frac{\kappa+ O(\log^{-1}u)}{\log y}\right) u (\log u +
\log_2(u+2)).\end{equation}
\end{lemma}

\section{Short character sums to prime moduli, Proof of Theorem \ref{small_prime_theorem}}
We dispose of quickly the case of small $N$, $\log N < \log_2^2 q
\log_3^{-10} q$
for both $\Delta(N, q)$ and $\Delta(\frac{q}{N},q)$.  The main work of Theorem
\ref{small_prime_theorem} will then be to consider the range $\log_2^2 q\log_3^{-10} q <
\log N =
o(\sqrt{\log q \log_2 q})$.

\subsection{Case of small $N$} When $\log N <
\log_2^2 q \log_3^{-10} q$ notice that for all $n \leq N \log q$, $d(n)
\ll \log_2^2 q \log_3^{-11} q$.  Thus choosing \[r(p) =
\left\{\begin{array}{lll} 1 - (\log_2 q)^{-2} && 1 < p \leq \frac{\log
q}{\log_2^5 q} \\ 0 && \text{otherwise}\end{array}\right.,\] we verify
\[\sum_p \log p \frac{r(p)^2}{1-r(p)^2} \leq \log_2^4 q \sum_{p < \frac{\log
q}{\log_2^5 q}} \log p \ll \frac{\log q}{\log_2 q}\]
and
\[\sum_{k \log p > \frac{\frac{1}{4} \log q}{\log_2^4 q}}
r(p)^{2k}p^{k
\frac{\log_2^2 q}{\frac{1}{4} \log q}} \ll \log q (1 - \log_2^2
q)^{\frac{2\log
q}{\log_2^6}} = o(1),\]
so that by Lemma \ref{tail_bound_lemma}, $\sum_{n \leq q^{\frac{1}{4}}} r(n)^2
= (1 + o(1)) \prod_p (1 - r(p)^2)^{-1}$.
Furthermore, for all $n < N \log q$ such that $p|n \Rightarrow p < \frac{\log
q}{\log_2^3 q}$ we have $r(n) \gg 1$.

Choosing $x = \frac{q}{N}$, we have $\sum_{n \leq \frac{x}{N}}r(n)^2 = (1 +
o(1)) \sum_n r(n)^2$, and so by Proposition \ref{fundamental_prop}
\[\Delta(N,q) \geq (1 + o(1)) \sum_{n < N} r(n) \geq (1 + o(1)) \Psi(N,
\frac{\log q}{\log_2^3 q}) \geq (1 + o(1))\Psi(N, \log q).\]

Choosing $x = \frac{1}{\log q} \sqrt{\frac{q}{N}}$ and setting $N' = N \log
q$, we have $\sum_{n \leq \frac{x}{N'}}r(n)^2 = (1 + o(1)) \sum_n r(n)^2$, and
so by Proposition \ref{fund_dual} we have
\begin{align*} \Delta(\frac{q}{N}, q) &\gg \sqrt{q}\sum_{n \leq N \log q}
\frac{r(n)}{n}(1 - c(\frac{n}{N}))\\&\gg \sqrt{q} \sum_{A = 1}^{\log q}
\sum_{\frac{N}{4} < h < \frac{3N}{4}}\frac{r(AN + h)}{AN + h} \gg
\frac{\sqrt{q}}{N} \sum_{A = 1}^{\log q} \frac{1}{A} \sum_{\frac{N}{4} < h <
\frac{3N}{4}} r(AN + h) .\end{align*}
It now follows as in  \cite{granville_soundararajan_large_char}, ('Proof of Theorem 11', p. 394) that
\[\Delta(\frac{q}{N}, q)  \gg \frac{\sqrt{q}}{N} \Psi(N, \log q) \frac{\log
\log q}{\log \left(\frac{\log N}{\log \log q}\right)}.\]
This completes the proof of Theorem \ref{small_prime_theorem} in the
case $\log N < \log_2^2 q \log_3^{-10}q$.    

\subsection{Main case}

Henceforth we assume that $\log_2^2 q \log_3^{-10} q < \log N = o(\sqrt{\log q
\log_2
q})$.  We are going to describe the analysis of $\Delta(N,q)$ in detail. 
Afterwards we will sketch the necessary modifications in order to handle the
dual case of $\Delta(\frac{q}{N}, q)$.  

Throughout the treatment of $\Delta(N,q)$  we fix $x = \frac{q-1}{N}$.  
Let $\epsilon =
\epsilon(q) > 0$ be a parameter tending to 0 as $q \to \infty$ and set $M =
(1-\epsilon)\log q$.  We let $\sigma = \sigma(N,q),$ $\frac{1}{2}+\frac{1}{\log
_2 q}< \sigma < 1$
be another parameter which will eventually be the location of the relevant
saddle point. We define completely multiplicative 'resonator' function
$r_\sigma(n)$ by
\[r_\sigma(p) = f_\sigma(\frac{p}{M})\] where $0<f_\sigma(x)<1$ is the unique
continuous solution to the equations
\[ \frac{f_\sigma(x)}{(1-f_\sigma(x)^2)^2} =
\left(\frac{c_\sigma}{x}\right)^\sigma, \qquad \qquad \int_0^\infty
\frac{f_\sigma(x)^2}{1-f_\sigma(x)^2}dx = 1.\]
Note that the second equation implicitly defines the constant $c_\sigma$.
The following basic properties of the function $f_\sigma$ may be established
with a little calculus.

\begin{lemma}\label{f_facts}
For each $\sigma \in (\frac{1}{2},1)$, the function $f_\sigma$ satisfies the
following properties.
\begin{enumerate}
 \item $f_\sigma$ is smooth, decreasing, and a bijection $(0,\infty) \to (0,1)$.
 \item $f_\sigma(x) \leq \left(\frac{c_\sigma}{x}\right)^\sigma$
 \item $\min\left(\frac{1}{2}, \frac{1}{4}
\left(\frac{x}{c_\sigma}\right)^{\frac{\sigma}{2}}\right) \leq 1 - f_\sigma(x)
\leq \left(\frac{x}{c_\sigma}\right)^{\frac{\sigma}{2}}$
\item The value of $c_\sigma$ is
\[c_\sigma =
\frac{\Gamma(\frac{3}{2\sigma})}{\Gamma(1 -
\frac{1}{2\sigma})\Gamma(\frac{2}{\sigma}-1)} \]
In particular, $c_\sigma <1$, $c_\sigma \to \frac{1}{2}$ as $\sigma \uparrow 1$
and $c_\sigma \sim (2\sigma - 1)$ as $\sigma \downarrow \frac{1}{2}$.
\item The Mellin transform of $f_\sigma$ is
\[\hat{f}_\sigma(s) =\int_0^\infty f_\sigma(x) x^{s-1}dx=
\frac{c_\sigma^s}{2\sigma}\frac{\Gamma(\frac{1}{2} -
\frac{s}{2\sigma})\Gamma(\frac{2s}{\sigma} + 1)}{\Gamma(\frac{3}{2} +
\frac{3s}{2\sigma})},\]
which converges absolutely in $0 < \Re(s) < \sigma$.  In particular
\[\hat{f}_\sigma(1-\sigma) = \frac{1}{2(1-\sigma)}c_\sigma^{-\sigma}.\]
\item The Mellin transform of $g_\sigma(x) =
\frac{f_\sigma(x)^2}{1-f_\sigma(x)^2}$ is
\[\hat{g}(s) = \int_0^\infty g(x)x^{s-1}dx = \frac{c_\sigma^s}{s}\frac{\Gamma(1
-\frac{s}{2\sigma})\Gamma(\frac{2s}{\sigma} -
1)}{\Gamma(\frac{3s}{2\sigma})}.\]  This integral converges absolutely in
$\frac{\sigma}{2} < \Re(s) < 2\sigma$.
\end{enumerate}

\end{lemma}
\begin{proof}
 The various integrals may be computed by substituting $\frac{dx}{x} =
-\left[\frac{4f}{1-f^2} + \frac{1}{f}\right]df.$
\end{proof}
\begin{lemma}\label{mellin_distort}
 Uniformly in $\frac{1}{2} < \sigma <1$ there is a constant $c>0$
such that 
\begin{align*}&\left|\frac{\hat{f}_\sigma(1-\sigma +
it)}{\hat{f}_\sigma(1-\sigma)}\right| \leq \left\{ \begin{array}{cll} \left(1 +
\frac{ct^2}{(1-\sigma)^2}\right)^{\frac{-1}{2}}&& |t|<
\frac{1}{4}\\
 1-c && |t| \geq \frac{1}{4}.\end{array}\right..
\end{align*}
\end{lemma}
\begin{proof}
 This follows on Taylor expanding $\hat{f}(s)$ about $s = 1-\sigma$.
\end{proof}

\subsection{Bound of tail for sum of squares, Proof of part
A of Theorem \ref{small_prime_theorem}}\label{tail_bound_1}
We apply Lemma \ref{tail_bound_lemma} with $y_i = y_q = \frac{x}{N}$. Recall
that we set $\alpha = \frac{\log y}{\log\log^2 y}$ and that we assume $\sigma >
\frac{1}{2} + \frac{1}{\log_2 q}$. 

For $\epsilon >
\frac{C}{ \log \log q}$, $C$ fixed but sufficiently large, we have uniformly in
$\frac{1}{2} < \sigma < 1$,
\begin{align*} \sum_p \log p \frac{r(p)^2}{1-r(p)^2} &\leq \sum_n \Lambda(n)
\frac{f_\sigma(\frac{n}{M})^2}{1-f_\sigma(\frac{n}{M})^2} \\&=
\frac{1}{2\pi i}\int_{(\frac{1}{2} + \sigma)}\frac{-\zeta'}{\zeta}(s)
M^s\hat{g}(s) ds
\\&\leq M(1 + O(\exp(-\sqrt{\log M}))) \leq \log \frac{x}{N} - \frac{\log
\frac{x}{N}}{\log \log q}\end{align*} by shifting the contour into the
standard zero-free region for zeta and passing the pole at 1 of
$-\frac{\zeta'}{\zeta}$.  Thus the first condition of the
lemma is satisfied.  Meanwhile,
\begin{align*} &\sum_{k \log p > \frac{\log \frac{x}{N}}{(\log \log q)^4}}
r(p)^{2k}p^{2k \alpha} \\&\ll \sum_{p <
\exp(\log^3 M)} \frac{(1 -(1-f_\sigma(\frac{p}{M})^2))^{\frac{\log
q}{\log_2^{8}q}}}{1 - f_\sigma(\frac{p}{M})^2} + \sum_{p > \exp(\log^3 M)}
\frac{f_\sigma(\frac{p}{M})^2
p^\alpha}{1-f_\sigma(\frac{p}{M})^2p^\alpha}.\end{align*} Substituting the
lower bound [(3) of Lemma \ref{f_facts}] for $1-f_\sigma$ in the first sum
and the upper bound [(2) of Lemma \ref{f_facts}] for $f_\sigma$ in the second
sum proves that each sum is $o(1)$, uniformly in $\sigma > \frac{1}{2} +
\frac{1}{\log_2 q}$.  This verifies the second condition of
Lemma \ref{tail_bound_lemma}.   It follows that \[\sum_{n \leq
\frac{x}{N}}r_\sigma(n)^2 = (1 + o(1)) \prod_p (1 - r_\sigma(p)^2)^{-1}, \qquad
q \to \infty,\] uniformly for $\sigma \in [\frac{1}{2} + \frac{1}{\log_2 q}, 1]
$. 

\subsection{Saddle point asymptotics, Proof of
part B of Theorem \ref{small_prime_theorem}}
We introduce the generating Dirichlet series
 \[R_\sigma(s) = \prod_p \left(1 -
\frac{r_\sigma(p)}{p^s}\right)^{-1},\; \Re(s) \geq 1-\sigma\]  Define its
logarithm
and logarithmic derivatives by
\[\phi_{0,\sigma}(s) = \log R_\sigma(s), \qquad \qquad \phi_{j,\sigma}(s) =
(-1)^j
\frac{d^j}{ds^j}\phi_0(s), \;\; j>0.\]
Explicitly, \[ \phi_{j,\sigma}(s) =  \sum_p \log^j p \sum_{n=1}^\infty
n^{j-1}\left(\frac{r_\sigma(p)}{p^s}\right)^n, \;\; (j \geq 0).\]
We prove the following proposition, which asymptotically evaluates $\sum_{n
\leq N} r_\sigma(n)$.
\begin{prp}\label{sum_evaluation_1}
 Let $\frac{1}{2} < \sigma < 1$ solve $\log N =
\phi_{1,\sigma}(\sigma)$.  Then  
\[\sum_{n \leq N}r_\sigma(n) \sim
\frac{N^{\sigma}e^{\phi_{0,\sigma}(\sigma)}}{\sigma
\sqrt{2\pi \phi_{2,\sigma}(\sigma)}}.\]
\end{prp}

In order to establish this proposition, we first require some bounds on
$R_\sigma(s)$
away from the real axis.  These are very similar to the bounds established in
\cite{hildebrand_tenenbaum} for $\zeta(s, y)$.

 It will be convenient to work with the
 'non-multiplicative' approximations
\begin{equation}\label{non_mult_approx} \tilde{\phi}_{j,\sigma}(s) = \sum_{n =
1}^\infty
\frac{\Lambda(n)
f_\sigma(\frac{n}{M})\log^{j-1}n}{n^s}.\end{equation}
Our first lemma demonstrates that $\tilde{\phi}_j$ is in fact a strong
approximation to $\phi_j$.  

\begin{lemma}\label{approx_mult}
 Let $s = \sigma + it$. Uniformly in $\frac{1}{2} + \frac{1}{\log_2 q} \leq
\sigma \leq 1 - \frac{1}{\log_2 q}$ we have \[\left|\phi_{j,\sigma}(s)-
\tilde{\phi}_{j,\sigma}(s)\right|
\ll \frac{\log^j M}{M^{\min(\frac{\sigma}{2}, \sigma - \frac{1}{2})}}.\]
\end{lemma}
\begin{proof}
We have
  \begin{align*}
  \left|\phi_{j,\sigma}(s)-
\tilde{\phi}_{j,\sigma}(s)\right|& = \left|\sum_p \log^j p \sum_{n = 2}^\infty
n^{j-1}
\frac{f_\sigma(\frac{p}{M})^n - f_\sigma(\frac{p^n}{M})}{p^{ns}}\right|\\
& \leq \sum_p \log^j p \sum_{n=2}^\infty n^{j-1} \frac{|1 -
f_\sigma(\frac{p}{M})^n| +
|1 - f_\sigma(\frac{p^n}{M})|}{p^{n\sigma}}\\
& \leq \sum_p \log^j p \left|1 - f_\sigma\left(\frac{p}{M}\right)\right|
\left(\sum_{n
= 2}^\infty \frac{n^j}{p^{n\sigma}}\right) + \sum_p \log^j p \sum_{n=2}^\infty
\frac{n^{j-1}}{p^{n\sigma}}\left|1 - f_\sigma\left(\frac{p^n}{M}\right)\right|,
 \end{align*}
and the claimed bound follows on substituting the upper bound in (3) of
Lemma \ref{f_facts}. 
\end{proof}
Our next lemma allows us to make explicit the relationship between $N$ and
$\sigma$ by evaluating $\phi_{1,\sigma}(\sigma)$.  
\begin{lemma}\label{phi_1_eval}
 We have uniformly in $\frac{1}{2}+
\frac{2}{\log_2 q}< \sigma < 1$, and $|t| <
\log_2 q$\[\phi_{1,\sigma}(\sigma+it) =
\hat{f}_\sigma(1-\sigma-it)M^{1-\sigma-it} -
\frac{\zeta'}{\zeta}(\sigma+it) + O(M^{1-\sigma}\exp(-\sqrt{\log M}))\]
In particular, for $\sigma$ solving $\phi_{1,\sigma}(\sigma) = \log N$,
the bounds  $\log_2^ q \log_3^{-10} q < \log N <\sqrt{\log
q} $ imply  $(1-\sigma)\gg \frac{\log_3 q}{\log_2 q}$ and $(\sigma -
\frac{1}{2}) \gg \frac{\log_3 q}{\log_2 q}$ as $q \to \infty$.
\end{lemma}
\begin{proof}
 It suffices to prove the corresponding result for $\tilde{\phi}_1$ since
the previous lemma implies that the resulting error is contained
in the error term.  To prove this lemma, write
\[\tilde{\phi}_1(\sigma+it) = \frac{1}{2\pi i}
\int_{(\frac{1}{2})}\left(-\frac{\zeta'}{\zeta}(s+\sigma+it)\right)M^s
\hat{f}_\sigma(s) ds,\] shift contours, and use the standard zero-free region. 
\end{proof}

\begin{lemma}\label{rel_size_derivs}
Suppose that $\sigma$ varies with $q$ in such a way that $(1-\sigma)\log_2 q \to
\infty$, $(2\sigma -1)\log_2 q \to\infty$ as $q \to
\infty$.  Then  
\[\phi_{j,\sigma}(\sigma) \sim_j \phi_{0,\sigma}(\sigma)\log_2^j q  .\]
\end{lemma}


\begin{proof}
Fix $Y$ such that $\log Y = \log M + \sqrt{\frac{\log M}{2\sigma -1}}$.  Then
\[\tilde{\phi}_{j,\sigma}(\sigma) = \sum_n
\frac{\Lambda(n)f_\sigma(\frac{n}{M})\log^{j-1}(n)}{n^\sigma} \leq
\tilde{\phi}_{0,\sigma}(\sigma) \log^j Y - \sum_{n >
Y}\frac{\Lambda(n)f_\sigma(\frac{n}{M})\log^{j-1}(n)}{n^\sigma} .\]
Since $f_\sigma(\frac{n}{M}) \leq \left(\frac{Mc_\sigma}{n}\right)^\sigma$ and
$\log Y \geq \frac{1}{2\sigma -1}$, the negative term is bounded by
\[ (Mc_\sigma)^\sigma \frac{\log^{j-1}Y}{(2\sigma-1)Y^{2\sigma -1}} \ll
\log^jY \hat{f}_\sigma(1-\sigma)\frac{M^{1-\sigma}}{\log M}
\cdot\left(\frac{M}{Y}\right)^{2\sigma-1}\ll  \log^j Y
\tilde{\phi}_0(\sigma)\left(\frac{M}{Y}\right)^{2\sigma -1},\] 
and this is $o(\tilde{\phi}_{0,\sigma}(\sigma))$.  In particular,
$\tilde{\phi}_{0,\sigma}(\sigma) \gg \frac{1}{\log M} \tilde{\phi}_{1,
\sigma}(\sigma).$

Now choose $Z$ so that $\log Z = \log M - \sqrt{\frac{\log M}{1-\sigma}}$.  Then
\[\tilde{\phi}_{j,\sigma}(\sigma) \geq \log^j Z
\left[\tilde{\phi}_{0,\sigma}(\sigma) - \sum_{n <
Z}\frac{\Lambda(n) f_\sigma(\frac{n}{M})}{n^\sigma \log n}\right].\] Bounding
$f_\sigma <1$ and using $\log Z \geq \frac{1}{1-\sigma}$, we have
\begin{align*}\sum_{n <
Z}\frac{\Lambda(n) f_\sigma(\frac{n}{M})}{n^\sigma \log n} \ll
\frac{Z^{1-\sigma}}{(1-\sigma)\log Z} &\ll \frac{M^{1-\sigma}}{(1-\sigma)\log
M}\left(\frac{Z}{M}\right)^{1-\sigma} \\&\ll
\frac{\tilde{\phi}_{1,\sigma}(\sigma)}{\log
M}\left(\frac{Z}{M}\right)^{1-\sigma} \ll
\tilde{\phi}_{0,\sigma}(\sigma)\left(\frac{Z}{M}\right)^{1-\sigma} =
o(\tilde{\phi}_{0,\sigma}(\sigma)).\end{align*}

Thus $\tilde{\phi}_{j,\sigma}(\sigma) \sim_j
\tilde{\phi}_{0,\sigma}(\sigma) \log^j M$.  But then applying Lemmas
\ref{approx_mult}
and \ref{phi_1_eval}, $\phi_{j,\sigma}(\sigma) \sim \tilde{\phi}_{j,
\sigma}(\sigma)$ for each
$j$.
\end{proof}

\begin{lemma}\label{phi_0_bound}
Let $\phi_{1,\sigma}(\sigma) = \log N$.  Then uniformly for $N$ varying in
the range $\log_2^2 q \log_3^{-10}q < \log N < \sqrt{\log q}$ we have 
\[
 \Re\left[\tilde{\phi}_{0,\sigma}(\sigma) - \tilde{\phi}_{0,\sigma}(\sigma +
it)\right] \gg
\left\{\begin{array}{lll} t^2 \tilde{\phi}_{2,\sigma}(\sigma), && |t| < (\log
Y)^{-1} \\
\frac{\phi_{1,\sigma}(\sigma)\min\left(\frac{t^2}{(1-\sigma)^2}, 1\right)}{\log
M +
\frac{3 \log_2
M}{2\sigma -1}} , && (\log Y)^{-1}  < |t| < M\end{array}\right.
\]
where $\log Y = \log M +
\sqrt{\frac{\log M}{2\sigma -1}}$.
\end{lemma}

\begin{proof}
 For all $t$ we have
\begin{equation}\label{phi_0_diff} \Re\left[\tilde{\phi}_{0,\sigma}(\sigma) -
\tilde{\phi}_{0,\sigma}(\sigma + it)\right] = \sum_n \frac{\Lambda(n)
f_\sigma(\frac{n}{M})}{n^\sigma \log n}(1 - \cos(t\log n)).\end{equation} For
$|t| <
(\log Y)^{-1}$ this is
\[ \gg t^2 \tilde{\phi}_{2,\sigma}(\sigma) - O(t^2 \sum_{n >
Y}\frac{\Lambda(n)f_\sigma(\frac{n}{M}) \log n}{n^\sigma}) \gg t^2
\tilde{\phi}_2(\sigma)\left(1 -
O\left(\left(\frac{M}{Y}\right)^{2\sigma -1}\right)\right),\] by bounding the
tail as in the previous lemma.  This is
$\gg t^2 \tilde{\phi}_2(\sigma)$ since $(2\sigma - 1)\log M \to \infty$.

For $(\log Y)^{-1} < |t| < M$ we have
\begin{align*}(\ref{phi_0_diff}) \geq& \frac{1}{\log Y}\Re\left\{\sum_{n}
\frac{\Lambda(n) f_\sigma(\frac{n}{M})}{n^\sigma}(1 - n^{-it}) - \sum_{n
> Y}\frac{\Lambda(n)f_\sigma(\frac{n}{M})}{n^\sigma}\right\}
\\ &\geq \frac{1}{\log Y} \Re\Biggl\{\hat{f}_\sigma(1-\sigma)M^{1-\sigma}\left(1
- \frac{\hat{f}_\sigma(1-\sigma -it)M^{-it}}{\hat{f}_\sigma(1-\sigma)}\right)
+\frac{\zeta'}{\zeta}(\sigma + it) - \frac{\zeta'}{\zeta}(\sigma)\\&
\qquad\qquad+ O\left(\phi_1(\sigma) \left(\frac{M}{Y}\right)^{2\sigma
-1}\right) + O\left(M^{1-\sigma} \exp(-\sqrt{\log M})\right)\Biggr\}.
\end{align*}
  Since 
$\hat{f}_\sigma(1-\sigma) \gg (1-\sigma)^{-1}$ we obtain
\begin{align*}\gg& 
\frac{\phi_1(\sigma) \min\left(\frac{t^2}{(1-\sigma)^2}, 1\right)}{\log
Z},\end{align*} by applying Lemma \ref{mellin_distort}.
\end{proof}

\begin{proof}[Proof of Proposition \ref{sum_evaluation_1}]
 Choose $T = \log N \log^2 M$, $\delta = 1-\sigma$ and apply Lemma
\ref{truncation} and the
second bound of Lemma \ref{phi_0_bound} to deduce that
\begin{align*}\sum_{n \leq N}r(n) = N^\sigma e^{\phi_{0,\sigma}(\sigma)}
&\Biggl\{\frac{1}{2\pi } \int_{-(1-\sigma)}^{1-\sigma} \exp(it \log N +
\phi_{0,\sigma}(\sigma + it) - \phi_{0,\sigma}(\sigma)) \frac{dt}{\sigma + it}\\
& \qquad\qquad +
O(\frac{1}{\log M \sqrt{\log N}}) + O\left(e^{ \frac{-c\log N}{\log M + \frac{3
\log_2 M}{2\sigma -1}}}\log \log N\right)\Biggr\}\end{align*}
Since $\phi_{2,\sigma}(\sigma) \sim \log N \log M$, these are genuine error
terms. 
Now $|\phi_{3,\sigma}(\sigma+it)| \leq \phi_{3,\sigma}(\sigma) \sim \log^2 M
\log N$ holds for
all $t$.  Splitting the integral accordingly at $|t| =
\log^{-2/3}M\log^{-1/3}N$,
and $|t| = \frac{1}{\log Y}$ with $Y$ as in Lemma \ref{phi_0_bound} we obtain
the main term by Taylor expanding $\phi_0$ on the interval around 0,
and error terms in the remaining part
of the integral.
\end{proof}
\subsection{Comparison to smooth number asymptotics, proof of part C of Theorem
\ref{small_prime_theorem}}
In this section we complete the proof of Theorem \ref{small_prime_theorem} in
the case $\log_2^2q  \log_3^{-10}q < \log N < \sqrt{\log q}$ by comparing
$\sum_{n \leq N} r_\sigma(n)$ to $\Psi(N, \kappa(\sigma) M)$. via the
following proposition.

\begin{prp}\label{comparison_prop} Let $\kappa(\sigma)$ be defined
by $\kappa^{1-\sigma} =
(1-\sigma)\hat{f}_\sigma(1-\sigma)$.  
Let, as before, $\log N = \phi_{1,\sigma}(\sigma)$.  We have
\[\left|\log\left(\frac{\sum_{n \leq N} r_\sigma(n)}{\Psi(N, \kappa(\sigma)
M)}\right)\right|
= O\left(\frac{\log N}{(2\sigma - 1)\log^2 M}\right).\]
\end{prp}

By Lemma \ref{log_Psi_change}, so long as $|\theta| <1$ and
$|\theta| \log_3 q\to \infty$ as $q \to \infty$, 
\[\log \frac{\Psi(N, e^\theta \kappa(\sigma)M)}{\Psi(N,\kappa(\sigma) M)} \sim
\frac{\theta}{\log M} \frac{\log N}{\log M}\log \frac{\log N}{\log M}.\]
Therefore, since
\[(2\sigma -1) \log\left(\frac{\log N}{\log M}\right) \to \infty\]
Proposition \ref{comparison_prop} establishes that there is some $\kappa' =
(1 +
o(1))\kappa$ for which $\Delta(N,q) \geq \Psi(N, \kappa' M) = \Psi(N, (1 +
o(1))\kappa \log q)$, which
completes the proof of Theorem \ref{small_prime_theorem} for $\Delta(N,q)$. 
Therefore, it suffices to prove Proposition \ref{comparison_prop}.

Recall that we defined $\psi_j(s;y) = \sum_{p < y} \log^{j-1}p \sum_{n =
1}^\infty \frac{1}{n^{1-j}p^{ns}}$.  The following
essential lemma allows us to
establish the approximation $\phi_{j,\sigma}(\sigma) \approx \psi_j(\sigma,
\kappa M)$,
$j = 1, 2$.

\begin{lemma}\label{phi_i_comparison}
 We have, uniformly in $\frac{1}{2}+ \frac{2}{\log M}< \sigma < 1$,
\begin{enumerate}
\item \[\phi_{1,\sigma}(\sigma) - \psi_1(\sigma, \kappa M) =
O(M^{1-\sigma}\exp(-\sqrt{\log M})).\]  
\item \[\phi_{0,\sigma}(\sigma) - \psi_0(\sigma, \kappa M) =
O\left(\frac{\phi_{1,\sigma}(\sigma)}{(2\sigma-1)\log^2 M}\right)\]
\end{enumerate}
\end{lemma}
\begin{proof}
In analogy with (\ref{non_mult_approx}) we introduce
\[\tilde{\psi}_j(s, \kappa M) = \sum_{n \leq \kappa M} \frac{\Lambda(n)
\log^{j-1}n} {n^{s}}.\]  In the range $\frac{1}{2} + \frac{1}{\log M} \leq
\sigma \leq 1- \frac{1}{\log M}$, the uniform bound
\[\left|\psi_j(\sigma, \kappa M) - \tilde{\psi}_j(\sigma, \kappa M) \right|
\ll  M^{\frac{1}{2} - \sigma}\log^j M\] is straightforward to establish along
the lines of Lemma \ref{approx_mult}.  Thus it suffices to prove the
corresponding statements of this lemma  for $\tilde{\phi}_{i,\sigma}$ and
$\tilde{\psi}_i$.  

To prove the first statement,  write
\[\tilde{\phi}_{1,\sigma}(\sigma) - \tilde{\psi}_1(\sigma, \kappa M) = 
\frac{1}{2\pi i}
\int_{(\frac{1}{2})}\left(-\frac{\zeta'}{\zeta}(s+\sigma)\right)M^s
\left[\hat{f}_\sigma(s)-\frac{\kappa^s}{s}\right] ds\] and note that the two
poles, at $s = 1-\sigma$ and at $s = 0$, are nullified by the difference. Shift
contours.

For the second, we have
\begin{align*}\tilde{\phi}_{0,\sigma}(\sigma) - \tilde{\psi}_0(\sigma, \kappa M)
&=
\sum_{n \leq \kappa M} \frac{\Lambda(n)(f_\sigma(\frac{n}{M})-1)}{n^\sigma \log
n} + \sum_{n \geq \kappa M} \frac{\Lambda(n)f_\sigma(\frac{n}{M})}{n^\sigma \log
n}\\
&= \frac{1}{\log M}\left[\sum_{n \leq \kappa M}
\frac{\Lambda(n)(f_\sigma(\frac{n}{M})-1)}{n^\sigma} + \sum_{n > \kappa M}
\frac{\Lambda(n)f_\sigma(\frac{n}{M})}{n^\sigma}\right] \\& \quad +
O\left(\sum_{n \leq \kappa M} \frac{\Lambda(n)
|1-f_\sigma(\frac{n}{M})|}{n^\sigma} \left|\log^{-1}n - \log^{-1}M\right|\right)
\\& \quad+ O\left(\sum_{n \geq \kappa M} \frac{\Lambda(n)
f_\sigma(\frac{n}{M})}{n^\sigma}\left|\log^{-1}n - \log^{-1}M\right|\right)
\end{align*}
The bracketed term is $\frac{1}{\log M}[\tilde{\phi}_{1,\sigma}(\sigma) -
\tilde{\psi}_1(\sigma, \kappa M)] = O(M^{1-\sigma}\exp(-\sqrt{\log M}))$, which
is permissible since $\phi_{1,\sigma}(\sigma) \sim \psi_1(\sigma, \kappa M) \sim
\frac{(\kappa M)^{1-\sigma}-1}{1-\sigma}$.  In the first error term, we  use $|1
- f_\sigma(\frac{n}{M})| \ll \left(\frac{n}{c_\sigma
M}\right)^{\frac{\sigma}{2}}$ for $n < c_\sigma M$, and bound trivially
$|1-f_\sigma(\frac{n}{M})|<1$ for $n > c_\sigma M$ to obtain
\[ \ll \frac{(c_\sigma M)^{1-\sigma}}{1-\frac{\sigma}{2}} \frac{|\log c_\sigma|
+ 1}{\log^2 M} + \frac{(\kappa M)^{1-\sigma} }{1-\sigma} \frac{|\log \kappa| +
1}{\log^2 M}.\]
In the second error term we bound $f_\sigma(\frac{n}{M}) \ll
\left(\frac{c_\sigma M}{n}\right)^\sigma$ to obtain
\[\ll \frac{(\kappa M)^{1-\sigma}}{2\sigma -1}
\left(\frac{c_\sigma}{\kappa}\right)^\sigma\left[\frac{|\log \kappa| +
1}{\log^2 M} + \frac{1}{(2\sigma -1) \log^2 M}\right].\]
Since $\frac{1}{2\sigma -1} \left(\frac{c_\sigma}{\kappa}\right)^\sigma$ is
bounded as $\sigma \downarrow \frac{1}{2}$ and $|\log c_\sigma|, |\log \kappa|
\ll \frac{1}{2\sigma -1}$, these errors are also permissible.
\end{proof}

We also need the following analogy of Lemma \ref{rel_size_derivs} for the
$\psi_j(\sigma, \kappa M)$.

\begin{lemma} Assume that as $q \to \infty$, $\sigma$ satisfies $\frac{1}{2}<
\sigma < 1$ and $(\sigma - \frac{1}{2}) \log M \to \infty$, $(1-\sigma)\log M
\to \infty$.  Then
\[ \psi_j(\sigma, \kappa M) \sim_j \log^j M \psi_0(\sigma, \kappa M).\]
\end{lemma}

\begin{proof}
 Observe \[\tilde{\psi}_j(\sigma, \kappa M) = \sum_{n \leq \kappa M}
\frac{\Lambda(n)\log^{j-1}n}{n^\sigma} \sim \frac{(\kappa
M)^{1-\sigma}\log^{j-1}M}{1-\sigma},\] so that $\tilde{\psi}_j(\sigma, \kappa
M) \sim_j \log^j M \tilde{\psi}_0(\sigma, \kappa M)$.  The statement follows,
since $\psi_j(\sigma, \kappa M) \sim \tilde{\psi}_j(\sigma, \kappa M)$ for all
$j$.
\end{proof}

\begin{proof}[Proof of Proposition \ref{comparison_prop}]
 Define $\tilde{N}$ by $\log \tilde{N} = \psi_1(\sigma, \kappa M)$.  Then
\[\left|\log\left(\frac{\sum_{n \leq N}r_\sigma(n)}{\Psi(N, \kappa
M)}\right)\right|
\leq \left|\log \left(\frac{\sum_{n \leq N}r(n)}{\Psi(\tilde{N}, \kappa
M)}\right)
\right| + \left|\log \left(\frac{\Psi(\tilde{N}, \kappa M)}{\Psi(N, \kappa
M)}\right)\right| = I + II.\]  

We first consider $I$.  By Proposition \ref{sum_evaluation_1} applied to
$\sum_{n \leq N} r_\sigma(n)$
and  and Theorem
\ref{smooth_saddle} applied to $\Psi(\tilde{N}, \kappa M)$ we have
\[I = \left|\sigma(\phi_{1,\sigma}(\sigma)- \psi_1(\sigma, \kappa M)) +
\phi_{0,\sigma}(\sigma)
- \psi_0(\sigma, \kappa M) - \frac{1}{2}\left[ \log \phi_{2,\sigma}(\sigma) -
\log
\psi_2(\sigma, \kappa M)\right]\right| + o(1).\] Since 
$\phi_{2,\sigma}(\sigma)\sim
\psi_2(\sigma, \kappa M) \sim \log N \log M$, substituting the bounds of
Lemma \ref{phi_i_comparison} gives
\[I = O\left(\frac{\log N
\log\left(\frac{\log N}{\log M}\right)}{(2\sigma -1)\log^2 M}\right).\]

For $II$, let $\alpha$ solve $\psi_1(\alpha, \kappa M) = \log N$ so that, by
Theorem \ref{smooth_saddle} applied to both $\Psi(\tilde{N}, \kappa M)$ and
$\Psi(N,
\kappa M)$,
\[II = \sigma \log \tilde{N} - \alpha \log N + \psi_0(\sigma, \kappa M) -
\psi_0(\alpha, \kappa M) - \frac{1}{2} \left[ \log \psi_2(\sigma, \kappa M) -
\log \psi_2(\alpha, \kappa M)\right] + o(1).\] Now by the Mean Value Theorem
\[|\log \tilde{N} - \log N| = \left|\psi_1(\sigma, \kappa M) - \psi_1(\alpha,
\kappa
M)\right| = |\sigma - \alpha|\cdot \psi_2(\gamma, \kappa M),\]\[|\psi_0(\sigma,
\kappa M) - \psi_0(\alpha, \kappa M)| = |\sigma - \alpha| \cdot\psi_1(\gamma',
\kappa M),\] for some $\gamma, \gamma'$ between $\alpha$ and $\sigma$. Moreover,
$\log N \sim \log \tilde{N}$ so \[\psi_2(\gamma, \kappa M) \sim \psi_2(\alpha,
\kappa
M) \sim \psi_2(\sigma, \kappa M) \sim \log N \log M\] and therefore 
\[|\sigma - \alpha| \ll \frac{|\phi_1(\sigma) - \psi_1(\sigma, \kappa M)|}{\log
N \log M} \ll \exp(-\sqrt{\log M}).\]
Combining these estimates, we find
\[II \ll \sigma |\log \tilde{N} - \log N| + |\sigma - \alpha| \log N + o(1) =
O(\frac{\log N}{\exp(-\sqrt{\log M})}),\] completing the proof.
\end{proof}

\subsection{Dual case $\Delta(\frac{q}{N}, q)$}
In this section we take $x = \frac{1}{\log q} \sqrt{\frac{q}{N}}$,  $M =
(\frac{1}{2}-\epsilon) \log q$ and define completely
multiplicative function $r_\sigma(n)$ by 
\[r_\sigma(p) = f_\sigma(\frac{p}{M}).\] 

A calculation that is exactly
analogous to the one in Section \ref{tail_bound_1} proves that for $\epsilon$
tending sufficiently slowly to 0,
\[\sum_{n \leq \frac{2x}{N}} r_\sigma(n)^2 = (1 + o(1)) \prod_p
(1-r_\sigma(p)^2)^{-1},\]  so that, by the second part of Proposition
\ref{fund_dual},
\[\Delta(\frac{q}{N}, q) \gg \frac{\sqrt{q}}{N^2} \sum_{n \leq \frac{N}{2}}
n r_\sigma(n).\]  

The evaluation of $\sum_{n \leq \frac{N}{2}} n r_\sigma(n)$
by a Perron integral yields
\[\frac{N}{2} \int_{\sigma - i \infty}^{\sigma + i \infty}
\left(\frac{N}{2}\right)^s R_\sigma(s) \frac{ds}{s+1};\] the only difference
between this integral and the one evaluated in Proposition
\ref{sum_evaluation_1} is the factor of $2^{-s}$ and the fact that $\frac{1}{s}$
has been replaced by $\frac{1}{s+1}$ in the denominator.  Exactly the same
method as there yields, for $\sigma$ solving $\phi_{1,\sigma}(\sigma) = \log N -
\log
2$, 
\[\sum_{n \leq \frac{N}{2}} n r_\sigma(n) \sim \frac{(N/2)^{1 +
\sigma}e^{\phi_{0,\sigma}(\sigma)}}{(1 + \sigma) \sqrt{2\pi
\phi_{2,\sigma}(\sigma)}}.\] It then
follows that for $\kappa'(\sigma)$ solving ${\kappa'}^{1-\sigma} =
(1-\sigma)\hat{f}(1-\sigma)$,
\[\Delta(\frac{q}{N},q) \gg \frac{\sqrt{q}}{N} \Psi(N, (1 + o(1))
\kappa'(\sigma) M).\]  Since $M = (1-o(1))\frac{\log q}{2}$, 
therefore
\[\Delta(\frac{q}{N},q) \gg \frac{\sqrt{q}}{N} \Psi\left(N, \left(\frac{1}{2} +
o(1)\right)\kappa(\sigma)\log q\right).\]



\section{Short sums to composite moduli, Proof of Theorem \ref{small_composite_theorem}}
Throughout let $M$ be minimal such that $\sum_{p\leq M}\log p > \log q$, and
define $u = \frac{\log N}{\log M}$.  Also, in this section $x =
\frac{\phi(q)}{N}$.  When $q$ is the product of many small
distinct primes the behavior of $\Delta(N,q)$ changes near where
$\log N = \log_2^3 q$.  Informally this is explained by the fact that for $N
< \exp(\log_2^3 q)$, most numbers less than $N$ but having all prime factors
larger than $\log q$ are composed of the same number of primes\footnote{This
number could be thought of as $\lfloor u \rfloor$, but
this is not quite accurate.}; for larger $N$ this is
no longer the case.

\subsection{Case $\log N = o(\log_2^2 q \log_3q)$}
Let $P>M$, $\log P \sim \log M$ be a parameter to be chosen, and let $\epsilon =
\frac{C}{\log_2 q}$ for a fixed sufficiently large constant $C$. Set $L =
\frac{\sqrt{MP}}{\log P} \frac{1}{(1 +
\epsilon)\sqrt{2}}.$ We define completely multiplicative $r(n)$ by
\[ r(p) = \left\{\begin{array}{lll} \frac{L \log p}{p}, && P < p < P^2\\ 0 &&
\text{otherwise}\end{array}\right.\]

\begin{rem}Intuitively we can understand the parameter $P$ as follows.  In order
to maximize $\sum_{n \leq N} r(n)$ we would like to have $r(n)$ be non-zero for
as my values $n$ as possible, but we would also like its value to be as large as
possible.  By increasing the starting point $P$ of the resonator we decrease
the number of $n$ for which $r(n)$ is non-zero, but for the remaining $n$ we
increase the value of $r(n)$.  $P$ will ultimately be chosen as a compromise
between these two competing factors. \end{rem}
\subsubsection{Bound for tail of squares, Proof of part A of Theorem
\ref{small_composite_theorem}}
We are going to apply Lemma \ref{tail_bound_lemma} to prove $\sum_{n \leq
\frac{x}{N}} r(n)^2 = (1 + o(1)) \sum_n r(n)^2$.  Thus $y = \frac{x}{N}$ and
$\alpha = \frac{\log_2^2 y}{\log y}$.

We have
\begin{align*}\sum_{P < p < P^2} \log p
\frac{r(p)^2}{1-r(p)^2} &\leq \frac{L^2}{1 - \frac{L^2 \log^2 P}{P^2}} \sum_{P <
p < P^2} \frac{\log^3 p}{p^2}\\&\leq (1 + O(\log^{-1} P))P \frac{ L^2 \log^2
P}{P^2 - L^2 \log^2 P}\\&\leq  (1 + O(\log^{-1} M) - \epsilon +
O(\epsilon^2))M \end{align*}
For $\epsilon = \frac{C}{\log_2 q}$ with $C$ sufficiently large, this
is $\leq \log \frac{x}{N} (1 - \log_2^{-1} \frac{x}{N})$, so that $r(n)$
satisfies the first condition of the Lemma.  Meanwhile, since $r(p) \leq
\frac{1}{\sqrt{2}}$, there is some fixed $c > 0$ such that
\[\sum_{P < p < P^2} \sum_{k \log p > \frac{\log y}{\log_2^4 y}}
r(p)^{2k}p^{k\alpha} \ll P^2 e^{-c \frac{\log y}{\log_2^5 y}} = o(1).\] 
Thus $r(n)$ also satisfies the second condition of Lemma \ref{tail_bound_lemma}.

\subsubsection{Evaluation of sum, Proof of part B of Theorem
\ref{small_composite_theorem}}
Before we evaluate the sum $\sum_{n \leq N} r(n)$ we introduce two more
parameters.  Let $\sigma (>\frac{1}{2})$ be the location of the saddle point in
the resulting Perron integral and set
\[\overline{u} = \frac{1}{(1+\epsilon)\sqrt{2}}\sqrt{\frac{M}{P}}
\frac{P^{1-\sigma}}{\sigma \log P};\] ultimately this will be chosen so that 
$\overline{u}\sim u$.  
The following Proposition characterizes our choice of
$P, \overline{u}$ and $\sigma$; they are taken to by any simultaneous solution
to the following system.

\begin{prp}\label{implicit_def}
 There exists a simultaneous solution $P, \overline{u}, \eta = \sigma \log P$
to the system of equations
\begin{enumerate}
 \item [a.] $\overline{u} = \frac{(MP)^{\frac{1}{2}}}{\eta e^\eta} \cdot
\frac{1}{(1 + \epsilon)\sqrt{2}}$
 \item [b.] $\overline{u} = \left\lfloor u \cdot \frac{\eta}{1 + \eta}\right
\rfloor$
 \item [c.] $\overline{u}
= \frac{\log N}{\log P} \cdot \frac{\eta}{1 + \eta} + \omega,$ for some
 $0 \leq \omega < \frac{2}{\log P},$
 \item [d.] $P>M$, $\log P \sim \log M$.
\end{enumerate}
Moreover, any solution to this system has $\eta = (1 + o(1)) \log M$.
\end{prp}
\begin{proof}
 Suppose b and c are satisfied.  Then
\[\frac{\log N}{\log P} = \left \lfloor \frac{\log N}{\log
M}\frac{\eta}{\eta+1}\right\rfloor \frac{\eta + 1}{\eta} - \frac{\eta +
1}{\eta} \omega \leq \frac{\log N}{\log M},\] so in fact, the first part of
condition d is redundant. Recall $u = O(\log^2 M \log \log M)$.  Combining a and
b, 
\[MP = \eta^2 e^{2\eta}2 (1 + \epsilon)^2 \left\lfloor u \frac{\eta}{1 +
\eta}\right\rfloor^2 \Rightarrow M \ll \eta e^\eta \log^2 M \log_2 M \] and so
$\eta \geq \log M - 4\log_2 M$.  Thus b and c now imply $\log P \sim \log M$,
so condition d may be completely discarded.  Furthermore, this guarantees that
at a solution $\eta \sim \log M$.

A solution may now be found as follows: beginning from $\eta = \eta_0 = \log M
- 4 \log_2 M$, increase the value of $\eta$ while defining $P = P(\eta)$ by
requiring
\[\left\lfloor u \cdot \frac{\eta}{1 + \eta}\right
\rfloor =\frac{(MP)^{\frac{1}{2}}}{\eta e^\eta} \cdot
\frac{1}{(1 + \epsilon)\sqrt{2}}.\] Clearly $P(\eta) \to \infty$ as $\eta \to
\infty$ and $\frac{\log N}{\log P}\cdot \frac{\eta}{\eta + 1} \to 0$.  Since
$\frac{\log N}{\log P(\eta_0)} \frac{\eta_0}{1 + \eta_0} > \left\lfloor
u\frac{\eta_0}{1 + \eta_0}\right\rfloor$, the proof of existence is completed by
checking that $\frac{\log N}{\log P(\eta)} $ jumps by at
most $\frac{2 + o(1)}{\log P}$ at discontinuities of the floor function.
\end{proof}

We need
two specific consequences of Proposition \ref{implicit_def}.

\begin{enumerate}
 \item [i.] $\frac{\log N}{\log P} - \overline{u} \cdot \left(1 +
\frac{1}{\sigma
\log P}\right)= o\left(\sqrt{\frac{\overline{u}}{\log \overline{u}}}\right)$
\item  [ii.] $\frac{\log N}{\log P} -\frac{\overline{u}}{\sigma \log P} \in \zed
+ o\left(\sqrt{\frac{\overline{u}}{\log \overline{u}}}\log^{-1}P\right)$.
\end{enumerate}
Note also that $\log P \sim \log M$ and  $\overline{u} \sim u =
O(\log^2 M \log_2 M)$ implies $\sigma = 1-o(1)$.

For the above choice of $P$ we are going to give the following evaluation for
$\sum_{n \leq N} r(n)$.
\begin{prp}\label{sum_evaluation}
 We have
\[\sum_{n \leq N} r(n) \geq (1 + o(1))\frac{N^\sigma e^{\overline{u}}} {2 \pi
\sigma \overline{u}}, \qquad \qquad u \to \infty.\]
\end{prp}

Before proceeding to the proof we introduce the generating function
\[R(s) = \prod_{p}(1 - r(p)p^{-s})^{-1}\] and  its logarithm 
\[\log R(s) = \phi_0(s)  = \sum_p \sum_{n = 1}^\infty
\frac{1}{n}\left(\frac{r(p)}{p^s}\right)^n\] and record several simple
properties.
\begin{lemma}\label{phi_0_determined}
Let $s = \sigma + it$.  Uniformly in $\frac{1}{2} \leq \sigma \leq 1$ and $|t|<
M$,
\[\phi_0(s) = \overline{u}\left\{\frac{P^{-it}}{1 + \frac{it}{\sigma}} +
O(\exp(-\sqrt{\log M}))\right\} = \overline{u}\frac{P^{-it}}{1 +
\frac{it}{\sigma}} + o(1).\]
\end{lemma}
\begin{proof}
This follows from the Prime Number Theorem.
\end{proof}
\begin{lemma}\label{phi_0_bound}
Let $\frac{1}{2} < \sigma < 1$ and set $\tau = 10 \left(\frac{\log
\overline{u}}{\overline{u}}\right)^{\frac{1}{2}}.$
\begin{enumerate}
 \item \[\Re\left[1 - \frac{P^{-it}}{1 + \frac{it}{\sigma}}\right] \geq \left\{
\begin{array}{lll} \frac{t^2}{5}, && |t| < \frac{1}{2}\\ \frac{1}{20} && |t| >
\frac{1}{2}\end{array}\right.\]
\item For $t = t_0 + \frac{2\pi j}{\log P}$, $j \in \zed$, $|j| < \frac{\tau
\log P}{2\pi}$ and $\frac{100 \tau}{\log P} < |t_0| < \frac{\pi}{\log P}$ we
have
\[\Re\left[1 - \frac{P^{-it}}{1 + \frac{it}{\sigma}}\right] \geq \frac{t_0^2
\log^2 P}{20}.\]
\end{enumerate}

\end{lemma}

\begin{proof}[Proof of Proposition \ref{sum_evaluation}]
 Applying Lemma \ref{truncation} with $T = \overline{u}^3$ we find
\begin{align*}\sum_{n \leq N} r(n) = &\frac{N^\sigma e^{\phi_0(\sigma)}}{2\pi
\sigma } \Biggl\{\int_{-\frac{1}{2}}^{\frac{1}{2}} \exp\left(it \log N +
\phi_0(\sigma + it) - \phi_0(\sigma)\right)\frac{dt}{1 + \frac{it}{\sigma}} \\&
\qquad\qquad + O(\overline{u}^{-3/2}) + O\left(\sup_{\frac{1}{2} < |t| <
\overline{u}^{3}} |\exp(\phi_0(\sigma + it) - \phi_0(\sigma))| \log
\overline{u}\right)\Biggr\}\end{align*}  On this range, Lemma
\ref{phi_0_determined} gives $\phi_0(\sigma + it) = \overline{u}\frac{P^{-it}}{1
- \frac{it}{\sigma}} + o(1)$, and so the bound of Lemma \ref{phi_0_bound} gives
that the second error term is $O(e^{-c \overline{u}}\log \overline{u})$, so
that both errors are permissible.  The main integral is thus
\[\int_{|t|< \tau + \frac{\pi}{\log P}} \exp\left(it \log N +
\overline{u}\left[\frac{P^{-it}}{1 + \frac{it}{\sigma}} - 1\right]\right)(1 +
o(1)) dt
+ O\left(\int_{\tau < |t| < \frac{1}{2}} \exp(-\overline{u}\frac{t^2}{10})
dt\right).\] 
Again the error is permissible.  Partitioning the integral into short intervals
we obtain
\begin{align*} &\sum_{|j| < \frac{\tau \log P}{2 \pi}} \int_{\frac{2\pi}{\log
P}(j-\frac{1}{2})}^{\frac{2\pi}{\log P}(j + \frac{1}{2})} \exp\left(it \log N +
\overline{u}\left[\frac{P^{-it}}{1
 + \frac{it}{\sigma}} - 1\right]\right)(1 + o(1)) dt
\\ & =\sum_{|j| < \frac{\tau \log P}{2 \pi}} \int_{-\frac{\pi}{\log
P}}^{\frac{\pi}{\log P}} \exp\left(i \log N\left(\frac{2\pi j}{\log P} 
+ t_0\right) + \overline{u}\left[\frac{e^{i t_0 \log P}}{1 +
\frac{i}{\sigma} \left(\frac{2\pi j}{\log P} + t_0\right)} -1\right]\right)(1 +
o(1)) dt_0\\
&=\frac{1}{\log P} \sum_{|j| < \frac{\tau \log P}{2\pi}} e\left(\frac{\log
N}{\log
P} j\right) \int_{|t_0| < 100 \tau} \exp\left(i \frac{\log N}{\log P}t_0 +
\overline{u}\left[\frac{e^{-it_0}}{1 + \frac{i(2\pi j + t_0)}{\sigma \log
 P}} - 1\right] \right)(1+ o(1))dt_0
\end{align*}
with an error of $o(\overline{u}^{-1})$ from truncating the integral (apply
Lemma \ref{phi_0_bound} (2)).  

We Taylor expand the inner bracket as
\begin{align*}& \left(1 - it_0 - \frac{t_0^2}{2} + O(\tau^3)\right)\left(1 -
\frac{i(2\pi j +
t_0)}{\sigma \log P} - \left(\frac{2\pi j + t_0}{\sigma \log P}\right)^2 +
O(\tau^3)\right) - 1
\\&= -i t_0 - i \frac{2\pi j + t_0}{\sigma \log P} - \frac{t_0^2}{2} -
t_0\left(\frac{2\pi j + t_0}{\sigma \log P}\right) - \left(\frac{2\pi j +
t_0}{\sigma \log P}\right)^2 + O(\tau^3)
\end{align*}
The conditions (i) and (ii) on $P, \sigma, \overline{u}$ have been specifically
made so that the linear oscillatory phases in both the sum and the integral are
now $o(1)$ throughout the range of integration.  Thus we obtain
\[\frac{1+o(1)}{\log P} \sum_{|j| < \frac{\tau \log P}{2\pi}} e^{-\overline{u}
\left(\frac{2\pi j}{\sigma \log P}\right)^2} \int_{|t_0| < 100 \tau}
e^{-\overline{u}\left(t_0^2\left(\frac{1}{2} - \frac{1}{\sigma \log P}\right) -
t_0 \left(\frac{2\pi j}{\sigma \log P} + \frac{2\pi j}{(\sigma \log
P)^2}\right)\right)} dt_0.\]  Estimating the integral, completing the sum to an
infinite one, and then applying Poisson summation yields that this is $\geq
\frac{1-o(1)}{\overline{u}}$.

\end{proof}

\subsubsection{Evaluation of saddle point asymptotics, Proof of part C of
Theorem \ref{small_composite_theorem}}
Recall the notation $u = \frac{\log N}{\log M}$, $u' = \frac{\eta}{1 +
\eta} u$, and that we set $\overline{u} = \lfloor u' \rfloor$.  

To prove the first part of Theorem 2 it only remains to
estimate $N^\sigma$. From Proposition \ref{implicit_def} we have [recall $\eta =
\sigma \log P$]
\[\frac{P^{1-\sigma}}{\sigma \log P}\left(\frac{M}{P}\right)^{\frac{1}{2}}
\frac{1}{(1+\epsilon)\sqrt{2}} = \frac{\log N}{\log P} \cdot
\frac{\eta}{1+\eta} + \omega\] with $0 \leq \omega < \frac{2}{\log P}$. Put
$\delta = 1-\sigma$.  Then
\[\delta \log P = \frac{1}{2}(\log P - \log M) + \log((1 + \epsilon) \sqrt{2})
+ \log(1-\delta) + \log \log N  + \log \frac{\eta}{\eta + 1} +
O(\log^{-1} N)\] from which it follows
\[\delta \log N \leq o(1) + \left[\frac{1}
{2} \log N\left(1- \frac{\log M}{\log P} \right) +\frac{\log N}{\log
P}\left(\log_2 N + \log((1 + \epsilon)\sqrt{2})\right)\right].\]
 Write \[\frac{\log M}{\log P} =
\frac{\log N}{\log P}\cdot \frac{\log M}{\log N} = \frac{\frac{\eta +
1}{\eta}\left(\overline{u} - \omega\right)}{u} =\frac{\lfloor u' \rfloor}{u'} -
O(\log^{-1} N).\]  Furthermore, 
\[\frac{\log N}{\log P} \leq (1 + O(\log^{-1}M))\overline{u}.\]  Putting this
together we deduce
\[N^\sigma e^{\overline{u}}\gg \frac{N^{\frac{1}{2} + \frac{\lfloor u'
\rfloor}{u'}}}{(\log N)^{\lfloor u' \rfloor}} \left(\frac{e}{\sqrt{2}
+o(1)}\right)^{\lfloor u'\rfloor}\] which proves the Theorem for small $N$.

\subsection{Case $\log N \gg \log_2^3 q \log_3 q$}
Recall that we set $M$ to be minimal such that $\sum_{p < M}\log p > \log q$,
and that we choose $x = \frac{\phi(q)}{N}$.

Let $\sigma> \frac{1}{2} + \frac{1}{\log_2 q}$ be a parameter and define
completely multiplicative function
$r_\sigma(n)$ by
\[ r_\sigma(p) = \left\{\begin{array}{lll} \frac{\lambda}{p^\sigma}, && M < p <
M^2\\ 0, && \text{otherwise}\end{array}\right.; \qquad 
\lambda = (1 -\epsilon)\sqrt{\frac{2\sigma - 1}{2\sigma}} M^\sigma,\] where
$\epsilon = \frac{C}{\log_2 q}$ for a sufficiently large constant $C$.

\subsubsection{Bound for tail of squares, Proof of part A of  Theorem
\ref{small_composite_theorem}}
As we have previously, we set $y = \frac{x}{N}$, $\alpha = \log^{-1}y \log_2^2
y$ and apply seek to apply Lemma \ref{tail_bound_lemma}. 

For $\epsilon = \frac{C}{\log_2 q}$ with $C$ sufficiently large,
\[\sum_{M<p<M^2} \log p \frac{r_\sigma(p)^2}{1-r_\sigma(p)^2} \leq
(1-\epsilon)(2\sigma - 1) \sum_{M < p < M^2} \frac{\log
p}{p^{2\sigma}} < \log y (1 - \log_2^{-1}y),\] so that the first condition of
the lemma is satisfied.
Meanwhile, for some constant $c > 0$,
\[ \sum_{M < p < M^2}\sum_{n > \frac{\log y}{\log p \log_2^4 y}}
f(p)^{2n}p^{\alpha n}\ll M^2 e^{-c \frac{\log y}{\log_2^5 y}} =
o(1),\] so that the second condition is also satisfied.  Thus Lemma
\ref{tail_bound_lemma} gives (uniformly in $\sigma$) that \[\sum_{n \leq
\frac{x}{N}} r_\sigma(n)^2 = (1 + o(1)) \sum_n r_\sigma(n)^2\] so that
$\Delta(N, q) \geq (1 + o(1)) \sum_{n \leq N} r_\sigma(n).$

\subsubsection{Asymptotic analysis of sum, Proof of part B of Theorem
\ref{small_composite_theorem}}
Introduce the generating function
\[ R(s) = R_\sigma(s) = \prod_p (1-r_\sigma(p)p^{-s})^{-1}\] and its logarithm
and derivatives
\[\phi_0(s) = \log R(s), \qquad \phi_j(s) = (-1)^j \frac{d^j}{ds^j}\phi_0(j),\;
j \geq 0.\]  We prove the following evaluation of $\sum_{n \leq N}r_\sigma(n)$.

\begin{prp}\label{sum_eval_theorem_2}
 Let $N$ satisfy $\log_2^3 q \log_3 q\ll \log N < \sqrt{\log q}\log_2^{-1} q$
and define $\sigma$ by\footnote{Note that the function
$\phi_1$ itself implicitly depends upon the parameter $\sigma$ through
$R_\sigma(s)$.} $\phi_1(\sigma) = \log N$.  We have
\[\sum_{n \leq N} r_\sigma(n) \sim \frac{N^\sigma \exp(\phi_0(\sigma))}{\sigma
\sqrt{2\pi \phi_2(\sigma)}}.\]  
\end{prp}

In order to prove this Proposition by the saddle point method we need the
following estimates.

\begin{lemma}\label{phi_evaluation_smooth_long}
 Let $\sigma$ satisfy $2\sigma - 1 > \frac{\log_2 M}{\log M}$.  Then uniformly
in $\sigma$,
\begin{enumerate}
 \item $\phi_j(\sigma) = (1 + O(M^{-\frac{1}{2}})) \sum_{M < p < M^2}
\frac{r_\sigma(p)
\log^j p}{p^\sigma}.$
 \item $\phi_j(\sigma) = (1 + o(1)) \log^j M \phi_0(\sigma)$
 \item $|\phi_3(\sigma + it)| \leq \phi_3(\sigma)$
 \item For $|t| < \frac{1}{2\log M}$,
\[ \Re[\phi_0(\sigma) - \phi_0(\sigma + it)] \gg t^2 \phi_2(\sigma)\]
 \item For $\frac{1}{2\log M} < |t| < M$,
\[\Re [\phi_0(\sigma) - \phi_0(\sigma + it)] \gg \phi_0(\sigma)
\min\left(\frac{t^2}{(2\sigma -1)^2}, 1\right).\]
\end{enumerate}

\end{lemma}
\begin{proof}
 The first three items are straightforward.  For (4) note
\[ \Re[\phi_0(\sigma) - \phi_0(\sigma + it)] \geq \sum_{M < p < M^2}
\frac{r_\sigma(p)}{p^\sigma} (1 - \cos(t\log p)) \gg t^2 \sum_{M<p<M^2}
\frac{r_\sigma(p)\log^2 p}{p^\sigma} \gg t^2 \phi_2(\sigma).\]

For (5), write
\begin{align*}\Re[\phi_0(\sigma) - \phi_0(\sigma + it)] &\geq
\frac{\lambda}{2\log M} \Re\left[ \sum_{M < p < M^2} \frac{\log
p}{p^{2\sigma}}(1 - p^{-it})\right]\\& \geq \frac{\lambda}{2\log M} \Re\left[
\sum_{M < n < M^2} \frac{\Lambda(n)}{n^{2\sigma}}(1 - n^{-it})\right] +
O(M^{\frac{1}{2} -\sigma})\\
& \geq  \frac{\lambda}{2\log
M}\Re\left[\left(\frac{M^{1-2\sigma}}{2\sigma -
1}-\frac{M^{2-4\sigma}}{2\sigma - 1}\right) - \left(\frac{M^{1-2\sigma
+it}}{2\sigma -1 + it} - \frac{M^{2-4\sigma + 2it}}{2\sigma - 1 +
it}\right)\right]
\\& \qquad\qquad +  O(\phi_1(\sigma)\exp(-\sqrt{\log M})) 
\end{align*}
Now the bracket is
\[ \geq \frac{M^{1-2\sigma}}{2\sigma -1} \left\{1 - \frac{1 +
M^{1-2\sigma}}{1-M^{1-2\sigma}} \left|\frac{1}{1 + \frac{it}{2\sigma -1
}}\right|\right\}\]
from which we deduce that the above is
\[\geq \frac{\phi_1(\sigma)}{2\log M} \left\{\frac{\min(t^2, (\sigma
- \frac{1}{2})^2)}{3(2\sigma - 1)^2} + O(M^{1-2\sigma}) + O(\exp(-\sqrt{\log
M/2}))\right\}\]\[ \gg \phi_0(\sigma) \min\left(\frac{t^2}{(2\sigma -1)^2},
1\right), \qquad \qquad |t| > \frac{1}{2\log M}.\]
\end{proof}

\begin{proof}[Proof of Proposition \ref{sum_eval_theorem_2}]
Apply Lemma \ref{truncation} with $T = \log N \log^2 M$ and $\delta = 2\sigma -
1$ to obtain
\begin{align*}\sum_{n \leq N}r_\sigma(n) &= \frac{N^\sigma
\exp(\phi_0(\sigma))}{\sigma} \Biggl\{\frac{1}{2\pi}\int_{-\delta}^\delta
\exp(it\log N + \phi_0(\sigma + it) - \phi_0(\sigma)) \frac{dt}{1 +
\frac{it}{\sigma}} \\& \qquad\qquad +O(\log^{\frac{-1}{2}}N \log^{-1}M)+
O\left(\log_2 q \exp(-c \phi_0(\sigma)\right) \Biggr\}.\end{align*} Since
$\phi_0(\sigma) \sim \frac{\log N}{\log M}$ while $\phi_2(\sigma) \sim \log N
\log M$, both error terms are permissible.  

Split the remaining integral at $|t| = \phi_3(\sigma)^{\frac{1}{3}} \sim
\log^{\frac{1}{3}}N \log^{\frac{2}{3}} M$ and at $|t| = \frac{1}{2\log M}$. 
In the interval near $t = 0$, Taylor expand $\phi_0(\sigma + it)$ to obtain the
main term of size $(2\pi \phi_2(\sigma))^{-\frac{1}{2}}$.  For
$\phi_3(\sigma)^{\frac{1}{3}} < |t| < \frac{1}{2\log M}$ use the bound (4) of
the previous lemma to obtain an error term.  On the remaining interval
$\frac{1}{2\log M} < |t| < \frac{1}{2\sigma - 1}$ use the bound (5); for $\log N
> C \log_2^3 q \log_3 q$ for a sufficiently large fixed constant $C$ this also
produces an error term.
\end{proof}

\subsubsection{Analysis of implicit parameters, Proof of part C of Theorem
\ref{small_composite_theorem}}
By hypothesis, $\phi_1(\sigma) = \log N$.  By Lemma
\ref{phi_evaluation_smooth_long} we have $\phi_0(\sigma) \sim \frac{\log
N}{\log_2 q}$ and $\phi_2(\sigma) \sim \log N \log_2 q$ so it remains to
determine the quantity $N^\sigma$.  Again by Lemma
\ref{phi_evaluation_smooth_long}
\[\phi_1(\sigma) = (1 + O(M^{-\frac{1}{2}})) \lambda \sum_{M < p < M^2}
\frac{\log p}{p^{2\sigma}} = (1 + O(\log^{-1} M))
\frac{M^{1-\sigma}}{\sqrt{2\sigma(2\sigma - 1)}}.\]
Set $A =(2\sigma - 1) \log
M$. Then $A$ satisfies
\[(1 + O(\log^{-1} M))A e^A = \frac{1}{2\sigma} \frac{\log q \log_2
q}{\log^2 N} = \frac{1}{2\sigma \tau^2}, \] where we have set $\log N = \tau
\sqrt{\log q
\log_2 q}$ with $\tau < \frac{1}{\sqrt{\log_2 q}}$.  
Hence $A = - \log (2 \sigma) + 2\log \tau^{-1} - \log \log \tau^{-1} - \log 2 +
o(1)$ and therefore
\[ N^\sigma = \sqrt{N} \exp\left(u \log \tau^{-1} - \frac{u}{2} \log \log
\tau^{-1} - u \log 2 - \frac{1}{2}u \log \sigma + u + o(u)\right).\]  Since
$e^{\phi_0(\sigma)} = e^{u + o(u)}$ we obtain the final asymptotic
\[\Delta(N,q) \geq N^{\frac{1}{2}}\exp\left(u(\log \tau^{-1} - \frac{1}{2}\log
\log \tau^{-1} - \log 2 - \frac{1}{2}\log \sigma + o(1))\right).\]  

Set $\log N
= (\log q)^{1-\sigma'}$, $\frac{1}{2}<\sigma' < 1$.  Then we have
$\tau^{-1} = \log^{\sigma'-\frac{1}{2}} q
\log_2^{\frac{1}{2}}q$ and so $\log \tau^{-1} = (\sigma' - \frac{1}{2} )
\log_2 q
+ \frac{1}{2}\log_3 q$ and $\log_2 \tau^{-1} = \log_3 q +
\log(\sigma' - \frac{1}{2} -
)$. In particular, $\sigma = \sigma' + o(1)$.  Thus 
\begin{align*}\Delta(N,q) &\geq N^{\sigma'} \exp\left(-u\left(\log 2 -1+
\frac{1}{2}\log(\sigma' - \frac{1}{2}) + \frac{1}{2} \log\sigma'  +
o(1)\right)\right)\\&= \frac{N}{(\log N)^u} \left(\frac{e +
o(1)}{\sqrt{2(2\sigma'-1)(\sigma')}}\right)^u.\end{align*}

\section{The range $\log\log N \sim \frac{1}{2} \log \log q,$ Proof of Theorem
\ref{moderate_theorem}}
In this section we set $\log N = \tau \log^{\frac{1}{2}} q
\log_2^{\frac{1}{2}} q$.  We assume that $\tau^{-1} \leq \log_2^{O(1)} q$ and we
may assume $\tau \ll\log_3 q$ since the case of larger $\tau$ is
contained in Theorem \ref{large_theorem}.  We handle the main case $\Delta(N,q)$
and the dual
case $\Delta(\frac{q}{N},q)$ ($q$ prime) simultaneously.  In the first place
we put  $x = \frac{\phi(q)}{N}$ and in the dual case we choose $x =
\sqrt{\frac{q}{N}} \log^{-1} q$.  

In either case, set $\lambda = \log^{\frac{1}{2}}x \log_2^{\frac{1}{2}} x$ and
define completely multiplicative function $r(n)$ by
\[r(p) = \left\{\begin{array}{lll} \frac{\lambda}{p^{\frac{1}{2}}\log p}&&
\lambda^2 < p < \exp(\log^2 \lambda)\\ 0 && \text{otherwise}\end{array}\right.
.\]

\subsection{Bound for tail of sum of squares, Proof of part A of Theorem
\ref{moderate_theorem}}
We apply Lemma \ref{tail_bound_lemma} with $y_i = \frac{x}{N}$.  Recall $\alpha
= \frac{\log_2^2 y_i}{\log y_i}$.  
\begin{align*}\sum_{\lambda^2 < p < \exp(\log^2 \lambda)} \log p
\frac{r(p)^2}{1-r(p)^2} &\leq (1 + O(\log_2^2 q))\lambda^2 \sum_{\lambda^2 < p <
\exp(\log^2 \lambda)} \frac{1}{p \log p}\\&\leq (1 + O(\log_2^2 q)) \lambda^2
\left[ (2 \log \lambda)^{-1} - \log^{-2}\lambda\right] \\&\leq \log x -
(1-o(1))\frac{\log x \log_3 x}{\log_2 x}.\end{align*}  Since $\log N =
O(\log^{\frac{1}{2} + \epsilon} x)$,  the first condition in Lemma
\ref{tail_bound_lemma} is satisfied. 

Meanwhile, for some $c > 0$,
\[\sum_{\lambda^2 < p < \exp(\log^2 \lambda)}\sum_{k > \frac{\log y}{\log p
\log_2^4 y}} r(p)^{2k}p^{\alpha k} \ll \exp(\log^2 \lambda) (\log \lambda)^{-c
\frac{\log y}{\log_2^6 y}} = o(1)\] so that the second condition is also
satisfied.  Thus by Lemma \ref{tail_bound_lemma}, 
\[\sum_{n \leq \frac{x}{N}} r(n)^2 = (1 + o(1)) \sum_n r(n)^2\] and therefore
\[\Delta(N,q) \geq (1 + o(1)) \sum_{n \leq N} r(n), \qquad \Delta(\frac{q}{N},
q) \gg \frac{\sqrt{q}}{N^2} \sum_{n \leq \frac{N}{2}} n r(n).\]

\subsection{Saddle point asymptotics, Proof of part B of Theorem
\ref{moderate_theorem}}
Define for $\Re(s) > 0$, $R(s) = \prod_p \left(1 - \frac{r(p)}{p^s}\right)^{-1}$
and the logarithm and derivatives
\[\phi_0(s) = \log R(s),  \qquad \phi_j(s) = (-1)^j \frac{d^j}{ds^j}\phi_0(s).\]
We have the following asymptotic expansions of $\sum_{n \leq N} r(n)$, $\sum_{n
\leq N} n r(n)$.
\begin{prp}\label{sum_expansion_3}
 Let $\sigma > 0$ be the unique solution to $\phi_1(\sigma) = \log N$.  We have
\[\sum_{n \leq N}r(n) \sim \frac{N^\sigma e^{\phi_0(\sigma)}}{\sigma \sqrt{2\pi
\phi_2(\sigma)}}, \qquad \sum_{n \leq N} n r(n) \sim \frac{N^{1 +
\sigma}e^{\phi_0(\sigma)}}{(1 + \sigma) \sqrt{2\pi \phi_2(\sigma)}}.\]
\end{prp}

The proof of Proposition \ref{sum_eval_theorem_2} is easily adapted to this
case, so we simply record the necessary estimates.
\begin{lemma}\label{phi_j_bounds}
 Uniformly in $\frac{1}{4} < \sigma < \frac{3}{5}$ We have the following
estimates regarding the functions $\phi_j$.
\begin{enumerate}
 \item $|\phi_3(\sigma + it)| \leq \phi_3(\sigma)$
 \item $\phi_2(\sigma) \ll \log N \log^2 \lambda$
 \item For $|t| < \log^{-2}\lambda$, \[\Re[\phi_0(\sigma) - \phi_0(\sigma + it)]
\gg t^2 \phi_2(\sigma).\]
 \item For $\log^{-2}\lambda < |t| <\lambda$, \[\Re[\phi_0(\sigma) -
\phi_0(\sigma + it)] \gg \min(t^2, 1) \lambda^{\frac{1}{2}}.\]
\end{enumerate}
\end{lemma}
\begin{proof}
 The first three items are straightforward, so we show the proof of (4).
We have
\begin{align*}\Re[\phi_0(\sigma) - \phi_0(\sigma + it)] &\geq \sum_{\lambda^2 <
p < \exp(\log^2 \lambda)} \frac{\lambda}{p^{2\sigma}\log p}\left[1 - \cos(t\log
p)\right]\\
& \geq \Re\left\{\frac{\lambda}{3\log^2 \lambda} \sum_{\lambda^2 < p <
\lambda^3} \frac{\log p}{p^{\frac{6}{5}}} \left[1 -
p^{-it}\right]\right\}\\& \gg \frac{\lambda^{\frac{3}{5}}}{\log^2 \lambda}
\left[1 - \left|\frac{1}{5} + it\right|^{-1}\right] \gg \lambda^{\frac{1}{2}}
\min(t^2, 1) .
\end{align*}
\end{proof}

We point out one simple consequence of Proposition \ref{sum_expansion_3}.  
\begin{cor}
We
have
\[\sum_{n \leq \frac{N}{2}} n r(n) \asymp \sum_{n \leq N} n r(n).\]  
\end{cor}
\begin{proof}
Let $\sigma'$ solve $\phi_1(\sigma') = \log N - \log 2$.  It evidently suffices
to check that $\phi_0(\sigma) - \phi_0(\sigma') = O(1)$.  But by the Mean Value
Theorem, \[\phi_0(\sigma) - \phi_0(\sigma') = (\sigma' - \sigma)
\phi_1(\gamma)\] for some $\gamma \in [\sigma, \sigma']$ and \[\sigma' -
\sigma = \frac{\phi_1(\sigma) - \phi_1(\sigma')}{\phi_2(\gamma')} = \frac{\log
2}{\phi_2(\gamma')}\] for some $\gamma' \in [\sigma, \sigma']$.  The claim now
follows because $\phi_2(\gamma') \geq \phi_2(\sigma') \geq \phi_1(\sigma')
\sim \phi_1(\gamma) \sim \log N$.
\end{proof}

\subsection{Evaluation of parameters, Proof of part C of Theorem
\ref{moderate_theorem}}
By (2) of Lemma \ref{phi_j_bounds}, $\phi_2(\sigma) = \log^{O(1)} q$, and
therefore Proposition \ref{sum_expansion_3} and its Corollary imply that for
$\sigma > 0$ solving $\phi_1(\sigma) = \log N$ we have
\[\Delta(N,q) = N^\sigma \exp\left(\phi_0(\sigma) + O(\log_2 q)\right)\]
and for $q$ prime,
\[\Delta(\frac{q}{N}, q) \gg \frac{\sqrt{q}}{N} N^\sigma
\exp\left(\phi_0(\sigma) + O(\log_2 q)\right).\] The error term will be
negligible, so  it suffices to determine $N^\sigma$ and
$\phi_0(\sigma)$.

We first show that for
$N$ such that $\sigma < \frac{1}{2} + \frac{1}{\log_2 q \log_3^2 q}$, 
\begin{align} \label{near_one_half}\Delta(N,q) &\geq \sqrt{N}
\exp\left((1 + o(1)) \sqrt{\frac{\log
q}{\log_2 q}}\right),\\
\notag \Delta(\frac{q}{N}, q) &\geq \sqrt{\frac{q}{N}}
\exp\left((1 + o(1)) \sqrt{\frac{\log q}{2\log_2 q}}\right), \qquad (q \text{
prime}).\end{align} 
Note that for $\sigma = \frac{1}{2} + \frac{1}{\log_2 q \log_3^2 q}$, 
\[\log N = \phi_1(\sigma) \asymp \lambda \sum_{\lambda^2 < p < \exp(\log^2
\lambda)} \frac{1}{p^{1 + \log_2^{-1}q\log_3^{-2}q}} \gg \sqrt{\log q\log_2
q}\log_4 q\] so that $\tau \gg \log_4 q$.  Since $A \tau \to 0$ and $A\tau'
\to 1$ as $\tau \to \infty$, the bounds (\ref{near_one_half}) verify the
theorem in this range.

When $\sigma < \frac{1}{2}$, since $\log N =
\phi_1(\sigma) \ll \sqrt{\log q \log_2 q}\log_3 q$ and $\phi_2(\alpha)$ is
decreasing in $\alpha > 0$ we have
\[\frac{1}{2} - \sigma \ll \frac{\sqrt{\log q \log_2 q}\log_3
q}{\phi_2(\frac{1}{2})}.\]  Meanwhile
\[\phi_2(\frac{1}{2}) \geq \sum_{\lambda^2 < p < \exp(\log^2 \lambda)}
\frac{\lambda \log p}{p} \gg \lambda \log^2 \lambda.\]  It follows that
$\frac{1}{2} - \sigma \ll \frac{\log_3 q}{\log_2^2 q}$ and therefore
\[(\frac{1}{2} - \sigma) \log N \ll \frac{\log_3 q}{\log_2 q} \sqrt{\frac{\log
x}{\log_2 x}}.\]  Meanwhile,
\[\phi_0(\sigma) \geq \phi_0(\frac{1}{2}) \geq \lambda \sum_{\lambda^2 \leq p <
\exp(\log^2 \lambda)} \frac{1}{p\log p} \geq (1 + o(1))\frac{\lambda}{2\log
\lambda} = (1 + o(1))\sqrt{\frac{\log x}{\log_2 x}}.\] Combining these
estimates we obtain (\ref{near_one_half}) for the case $\sigma < \frac{1}{2}$.  

In the range $\frac{1}{2} \leq \sigma < \frac{1}{2} + \log_2^{-1} q \log_3^{-2}
q$ we have
$\phi_0(\frac{1}{2}) - \phi_0(\sigma) \leq (\sigma - \frac{1}{2})
\phi_1(\frac{1}{2}).$  Now
\[\phi_1(\frac{1}{2})  \sim \lambda \sum_{\lambda^2 < p < \exp(\log^2 \lambda)}
\frac{1}{p} = O(\lambda \log_2 \lambda)\] and therefore
\[\phi_0(\frac{1}{2})- \phi_0(\sigma) \ll \frac{1}{\log_3 q} \sqrt{\frac{\log
x}{\log_2 x}}.\]  Thus we also have (\ref{near_one_half}) for $\sigma <
\frac{1}{2} + \frac{1}{\log_2 q \log_3^2 q}$.

We now consider the case $\log_2^{-1} q \log_3^{-2} q < \sigma - \frac{1}{2} <
\log_2^{-1}q \log_3^2 q.$  In this range we have
\[\log N = \phi_1(\sigma) = (1 + O(\lambda^{-1})) \lambda \sum_{\lambda^2 < p <
\exp(\log^2 \lambda)} \frac{1}{p^{\frac{1}{2} + \sigma}}.\]  By the prime
number theorem and partial summation, this is
\[(1 + O(\exp(-\sqrt{\log \lambda}))) \int_{2(\sigma - \frac{1}{2}) \log
\lambda}^{(\sigma - \frac{1}{2}) \log^2 \lambda} e^{-x} \frac{dx}{x} = (1 +
O(\exp(-\sqrt{\log \lambda}))) \lambda \int_{(2\sigma - 1)\log \lambda}^\infty
e^{-x} \frac{dx}{x}.\]  Since $\tau \sim \frac{\log N}{\lambda}$ in the
case of $\Delta(N,q)$ and $\tau \sim \frac{\log N}{\sqrt{2}\lambda}$ in
the dual case of $\Delta(\frac{N}{q},q)$, the range $\log_2^{-1} q \log_3^{-2} q
< \sigma - \frac{1}{2} <
\log_2^{-1}q \log_3^2 q$ corresponds to $\tau$ within the range $e^{-O(\log_3^2
q)}\leq \tau \ll \log_4 q $, so the range of $\sigma$
considered is sufficient for the theorem.

Let $A$ solve $\frac{\log N}{\lambda} = \int_A^\infty e^{-x} \frac{dx}{x}$ and
let $A' = (2\sigma - 1)\log \lambda$.  It is readily seen that $|A-A'| \leq
\exp(-\frac{1}{2}\sqrt{\log \lambda})$ and therefore 
\[\sigma \log N = \frac{1}{2} \log N + \frac{A\log N}{2\log \lambda} + O(\log N
\exp(-\frac{1}{2} \sqrt{\log \lambda})).\]
Meanwhile, applying the prime number theorem a second time
\[\phi_0(\sigma) = (1 + o(1)) \lambda \sum_{\lambda^2 < p < \exp(\log^2
\lambda)} \frac{1}{p^{\frac{1}{2} + \sigma} \log p} = (1 + o(1))\lambda
\int_{2\log \lambda}^{\log^2 \lambda} e^{(\frac{1}{2} -
\sigma)y}\frac{dy}{y^2}\] 
and after a change of variables, this is
\[(1 + o(1))(\sigma - \frac{1}{2}) \lambda \int_{2(\sigma - \frac{1}{2}) \log
\lambda}^\infty e^{-x} \frac{dx}{x^2} = (1 + o(1))(\sigma -
\frac{1}{2})\lambda \int_A^\infty e^{-x} \frac{dx}{x^2}.\]
Since $A = (1 +
o(1)) (\sigma - \frac{1}{2})2 \log \lambda$ and recalling that we set $\tau'=
\int_A^\infty e^{-x} \frac{dx}{x^2}$ we obtain $\phi_0(\sigma) = (1 + o(1))
\frac{A \tau' \lambda}{2\log \lambda}$.  

Thus
\[N^\sigma e^{\phi_0(\sigma)} = N^{\frac{1}{2}}\exp\left((1 + o(1))\left(\frac{A
\log N}{2 \log \lambda} + \frac{A \tau'\lambda}{ 2\log
\lambda}\right)\right),\] and therefore,
\begin{align*} \Delta(N,q) &\geq \sqrt{N} \exp\left((1 + o(1)) \left(A(\tau +
\tau') \sqrt{\frac{\log q}{\log\log q}}\right) \right)\\
\Delta(\frac{q}{N},q) &\gg \sqrt{\frac{q}{N}} \exp\left((1 + o(1))\left(A
\left(\tau  + \frac{\tau'}{\sqrt{2}}\right) \sqrt{\frac{\log q}{\log_2
q}}\right)\right), \qquad (q \text{ prime}).\end{align*}

\section{Long character sums, Theorem \ref{large_theorem}}
When $\log N \geq 4\sqrt{\log q \log_2 q}\log_3 q$ we use a `second moment'
version of the resonance method, which avoids saddle point analysis.  The
situation now becomes similar to that in the original paper \cite{sound_resonance}.

The second moment version of the Fundamental Proposition is as follows.
\begin{prp}[Fundamental Proposition, second moment version]\label{fund_second}
 Let $\log^4 q < N < q $ and set $x =
\frac{q}{N}$. Let $r(n)$ be a non-negative multiplicative function supported on
squarefree $n$ and such that $p|q \Rightarrow r(p) = 0$.  Then for any
parameter $z < \frac{N}{\log^3 q}$ we have
\[\Delta(N,q)^2+O\left(\frac{\phi(q)}{q} N\right)  \gtrsim \frac{\phi(q)}{q} N
\sum_{\substack{n_1, n_2 \leq z\\ (n_1, n_2) = 1}}
\frac{r(n_1)r(n_2)}{\max(n_1,n_2)} \left[\sum_{\substack{g \leq
\frac{x}{\max(n_1,n_2)}\\ (g,n_1n_2) = 1}}r(g)^2 \bigg/ \prod_p\left(1 +
r(p)^2\right)\right] .\]
Moreover, let $M$ be minimal such that $\sum_{p \leq M}\log p > \log q$.  The
conclusion remains valid if the condition $p|q  \Rightarrow r(p) = 0$ is
replaced with $p \leq M \Rightarrow r(p) = 0$.
\end{prp}

\begin{proof}
Define 'resonator' $R(\chi) = \frac{1}{\sqrt{\phi(q)}}\sum_{n \leq x} r(n)
\chi(n)$.  Plainly
\begin{equation}\label{basic_ratio}\Delta(n,q)^2 \geq \sum_{ \chi \neq \chi_0}
|R(\chi)|^2 \left|\sum_{n \leq N} \chi(n)\right|^2 \Bigg/ \sum_{\chi}
|R(\chi)|^2.\end{equation}
By orthogonality of characters, the denominator is $\sum_{n \leq x} r(n)^2 \leq
\prod_p (1 + r(p)^2).$
Meanwhile, the numerator is
\[\sum_{\substack{m_1,m_2 \leq x}}r(m_1)r(m_2) \sum_{\substack{ n_1, n_2 \leq
N\\ (n_1n_2, q) = 1\\ m_1n_1 = m_2n_2}}  - |R(\chi_0)|^2\left|\sum_{n
\leq N} \chi_0(n)\right|^2.\]
By Cauchy-Schwartz, $|R(\chi_0)|^2 \leq \frac{x}{\phi(q)} \sum_{m \leq x}
r(m)^2$ and since $N > \log^4 q$, $\sum_{n \leq N} \chi_0(n) \sim
\frac{\phi(q)}{q} N.$  Therefore the negative term contributes
$O(\frac{\phi(q)}{q} N)$ to the ratio (\ref{basic_ratio}).  In the main term,
let $g = (m_1,m_2)$, $h = (n_1,n_2)$ and replace $m_i := \frac{m_i}{g}$ to
obtain
\[ \sum_{\substack{m_1, m_2 \leq x \\ (m_1,m_2) = 1}} r(m_1)r(m_2)
\sum_{\substack{g \leq \frac{x}{\max(m_1,m_2)}\\ (g,m_1m_2) = 1}} r(g)^2
\sum_{\substack{ h < \frac{N}{\max(m_1,m_2)}\\ (h,q) = 1}} 1.\] Discarding some
non-negative terms, this is
\[ 
\gtrsim \frac{\phi(q)}{q} N \sum_{\substack{m_1, m_2 \leq z \\ (m_1,m_2) = 1}}
\frac{r(m_1)r(m_2)}{\max(m_1,m_2)}
\sum_{\substack{g \leq \frac{x}{\max(m_1,m_2)}\\ (g,m_1m_2) = 1}} r(g)^2, \]
which proves the first part of the Proposition.  

For the second statement, let $r$ be any non-negative multiplicative function
supported on squarefree numbers and satisfying $p \leq M \Rightarrow r(p) = 0$. 
Enumerate $\{q_1, ..., q_r\}$ the set of primes greater than $M$ that divide
$q$, and $\{p_1, ..., p_s\}$ the set of primes at most $M$ that do
not divide $q$.  Then $s \geq r$ so we may define a new multiplicative function
$\tilde{r}$, supported on squarefrees and satisfying $p|q \Rightarrow
\tilde{r}(p) = 0$, by exchanging the values of $r(p_i)$ and $r(q_i)$ for $1
\leq i \leq r$.  Evidently 
\[\prod_{p} (1 + \tilde{r}(p)^2) = \prod_p (1 + r(p)^2)\]
and also, for any $z$,
\[\sum_{\substack{m_1, m_2 \leq z\\ (m_1,m_2) = 1}}
\frac{r(m_1)r(m_2)}{\max(m_1,m_2)} \sum_{\substack{g <
\frac{x}{\max(m_1,m_2)}\\ (g,m_1m_2) = 1}} r(g)^2 \leq  \sum_{\substack{m_1, m_2
\leq z\\ (m_1,m_2) = 1}}
\frac{\tilde{r}(m_1)\tilde{r}(m_2)}{\max(m_1,m_2)} \sum_{\substack{g <
\frac{x}{\max(m_1,m_2)}\\ (g,m_1m_2) = 1}} \tilde{r}(g)^2,\] which reduces the
second statement in the Proposition to the first case. 
\end{proof}

In the case $q$ is prime we also have a dual version of the Second Moment
Propostion.  Recall that we set
\[S(\chi) = \sum_{|h| \leq H, h \neq 0} \frac{\chi(h)}{h} (1 - c(\frac{h}{N})),
\qquad \qquad H = \sqrt{Nq}\log q,\]
and from P\'{o}lya's Fourier expansion, 
\[\left|\sum_{n \leq N} \chi(n)\right| \geq \frac{\sqrt{q}}{2\pi} |S(\chi)| +
O\left(\sqrt{\frac{q}{N}}\right),\] also $\chi(1) = 1 \Rightarrow S(\chi) = 0$
so that, in particular, $S(\chi_0) = 0$.

\begin{prp}[Fundamental Proposition, dual second moment version]\label{fund_second_dual}
Let $q$ be prime,  $N \leq \frac{q}{\log^2 q}$ and set $x =
\frac{1}{2\log q}\sqrt{\frac{q}{N}}$.  Let $r(n)$ be a multiplicative function
supported on squarefree numbers not divisible by $q$.  Then
\[\sup_{\chi \neq \chi_0} |S(\chi)|^2 \gg \frac{1}{N} \sum_{\substack{m_1, m_2
\leq \min(x, \frac{N}{2}) \\ (m_1,m_2) = 1}} \frac{r(m_1)r(m_2)
m_1m_2}{\max(m_1,m_2)^3} \sum_{\substack{g \leq \frac{x}{\max(m_1,m_2)}\\
(g,m_1m_2) = 1}} r(g)^2 \bigg/ \prod_p (1 + r(p)^2).\] 
\end{prp}
\begin{proof}
 Set $R(\chi) = \frac{1}{\sqrt{\phi(q)}} \sum_{m < x }r(m)\chi(m)$.  Plainly
\[\sup_{\chi\neq \chi_0} |S(\chi)|^2 \geq \sum_{\chi} |R(\chi)S(\chi)|^2 \bigg/
\sum_{\chi} |R(\chi)|^2.\]  The denominator is bounded by $\prod_p (1 +
r(p)^2)$ while the numerator is equal to
\[\sum_{m_1,m_2 \leq x} r(m_1)r(m_2) \sum_{\substack{0 \neq |n_1|,|n_2| \leq
H\\ m_1n_2 \equiv m_2n_1 \bmod q}} \frac{(1 -
c(\frac{n_1}{N}))(1-c(\frac{n_2}{N}))}{n_1n_2}.\] The congruence modulo $q$ is
possible only if $n_1$ and $n_2$ have the sign.  Discarding those (positive)
terms with $\max(n_1, n_2) > \frac{N}{2}$, the numerator is
\begin{align*}\gg& N^{-4} \sum_{m_1, m_2 \leq x} r(m_1)r(m_2) \sum_{\substack{0
< n_1,n_2 \leq \frac{N}{2} \\ n_1m_2 = n_2m_1}} n_1n_2 \\&=
N^{-4}\sum_{\substack{m_1,m_2 \leq x\\ (m_1,m_2) = 1}} r(m_1)r(m_2)m_1m_2
\sum_{\substack{g\leq \frac{x}{\max(m_1,m_2)}\\ (g, m_1m_2) = 1}} r(g)^2 \sum_{
h \leq \frac{N}{2 \max(m_1,m_2)}} h^2
\\& \gg \frac{1}{N} \sum_{\substack{m_1, m_2
\leq \min(x, \frac{N}{2}) \\ (m_1,m_2) = 1}} \frac{r(m_1)r(m_2)
m_1m_2}{\max(m_1,m_2)^3} \sum_{\substack{g \leq \frac{x}{\max(m_1,m_2)}\\
(g,m_1m_2) = 1}} r(g)^2.\end{align*}

\end{proof}

\subsection{Choice of resonator, and some lemmas}
Let either $x =\frac{q}{N}$ in the main case of $\Delta(N,q)$ or $x = \frac{1}{
2\log q} \sqrt{\frac{q}{N}}$ in the dual case of $\Delta(\frac{N}{q},q)$.  In
either case, set $\lambda = \sqrt{\log x \log \log x}$ and as in
\cite{sound_resonance}, define
multiplicative function $r(n)$ at prime powers by
\begin{align*} &r(p) = \left\{\begin{array}{cll} \frac{\lambda}{\sqrt{p}\log p}
&& \lambda^2
\leq p
\leq
\exp((\log \lambda)^2)\\ 0 && \text{otherwise}\end{array}\right.\\
&r(p^n) = 0, \qquad n \geq 2.\end{align*}
We also define multiplicative function $t$ by $t(p^n) = \frac{r(p^n)}{1 +
r(p^n)^2}$.

The following two estimates are extrapolated from those used in the
proof of \cite{sound_resonance} Theorem 2.1.
\begin{lemma}\label{ingredients}
 Assume $z > \exp(3\lambda \log \log \lambda)$.  As $x \to \infty$ we have
\begin{equation}\label{main_evaluation} \sum_{m \leq z} \frac{t(m)}{\sqrt{m}}
\sim \prod_p \left(1 + \frac{t(p)}{\sqrt{p}}\right) =  \exp\left((1 + o(1)) \frac{\lambda}{2\log \lambda}\right).
\end{equation}
Also, for $\alpha = \frac{1}{(\log \lambda)^3}$, 
\begin{align}\notag x^{-\alpha} \sum_{\substack{m_1, m_2 \leq z\\ (m_1, m_2) =
1}}\frac{r(m_1)r(m_2)}{(m_1m_2)^{\frac{1}{2}-\alpha}}\sum_{(d, m_1m_2) =
1}r(d)^2d^\alpha &\Bigg / \left(\sum_{m \leq x}
\frac{t(m)}{\sqrt{m}}\right)^2
\sum_d r(d)^2 \\ \label{g_tail}&\leq \exp\left(-(1 + o(1))\frac{32 \log x}{(\log
\log
x)^4}\right).
\end{align}
\end{lemma}

\begin{proof} First (\ref{main_evaluation}): By 'Rankin's trick', for any
$\alpha>0$ the sum is
\[\prod_p \left(1 + \frac{t(p)}{\sqrt{p}}\right) + O\left(z^{-\alpha}\prod_p
\left(1 + \frac{t(p)}{p^{\frac{1}{2} - \alpha}}\right) \right).\]
The main term is of the desired size (\cite{sound_resonance},
 top of page 6) so it suffices to bound the error. Choose $\alpha =
\frac{1}{(\log \lambda)^3}$. Then the ratio of the error term to the main term
is
\[\ll z^{-\alpha} \prod_p 
\left(1 + \frac{t(p)}{p^{\frac{1}{2}-\alpha}}\right) \bigg/\left(1 +
\frac{t(p)}{p^{\frac{1}{2}}}\right) \leq z^{-\alpha}\exp\left( \sum_p
\frac{t(p)}{\sqrt{p}} \left(p^\alpha - 1\right)\right),\] and since $\alpha
\log p \leq \frac{1}{\log \lambda}$ for all $p$ with $t(p) \neq 0$, this last is
bounded by
\[\leq \exp\left(\alpha\left(-\log z + 2\sum_{p} \frac{t(p) \log
p}{\sqrt{p}}\right)\right) \leq \exp\left(\alpha\left(-\log z + 2\lambda \sum_{
p
 < e^{(\log \lambda)^2}}\frac{1}{p}\right)\right) = o(1).\]

Now (\ref{g_tail}): By our calculation for (\ref{main_evaluation}), the
denominator of (\ref{g_tail}) is 
\[(1 + o(1)) \prod_p \left(1+ \frac{t(p)}{\sqrt{p}}\right) \prod_p \left(1 +
r(p)^2\right).\]
Meanwhile, the double sum in the numerator is bounded by
\[ \leq \sum_{\substack{m_1, m_2 \leq z\\ (m_1, m_2) =
1}}\frac{t(m_1)t(m_2)}{(m_1m_2)^{\frac{1}{2}-\alpha}} \sum_{d} r(d)^2 d^\alpha
\leq \left(\prod_p \left(1 + \frac{t(p)}{p^{\frac{1}{2}-\alpha}}\right)\right)^2
\prod_p(1 + r(p)^2 d^\alpha).\]
Therefore the ratio in
(\ref{g_tail}) is bounded by
\begin{align*}&\ll x^{-\alpha} \prod_p \left(\frac{1 + t(p) p^{\frac{-1}{2} +
\alpha}}{1 + t(p)p^{\frac{-1}{2}}}\right)^2 \prod_p \left(\frac{1 +
r(p)^2p^\alpha}{1 + r(p)^2}\right)
\\ & \ll \exp\left( -\alpha \log x  + 2\alpha \sum_{\lambda^2 \leq p \leq
\exp((\log
\lambda)^2)} \frac{r(p)\log p}{\sqrt{p}} + \alpha \sum_{\lambda^2 \leq p \leq
\exp((\log
\lambda)^2)} r(p)^2 \log p\right).
\end{align*}
Substituting the definition of $r(p)$ and $\alpha$, and using the prime number
theorem, the last expression is bounded by $\exp(\frac{-\lambda^2(1 +
o(1))}{\log^5
\lambda})$.
\end{proof}

We record one more elementary estimate.

\begin{lemma}\label{elementary}
Uniformly in $y \geq 1$ and for any $k > 0$,
\begin{equation}\label{mobius_sum}
 \sum_{\substack{d \leq y\\ (d,k) = 1}} \frac{\mu(d)t(d)^2}{d}
 = 1+o(1), \qquad \qquad(x \to \infty).
\end{equation}

\end{lemma}

\begin{proof}

The sum is 
 \begin{align*}
  1 + \sum_{\substack{\lambda^2 <  d <y\\ (d,k)=1}} \frac{\mu(d)t(d)^2}{d}
 &= 1 + O\left(\lambda^2{\sum_{d>\lambda^2}}^\flat
\frac{1}{d^2 (\log d)^2}\right) = 1 + o(1).
 \end{align*}
\end{proof}

Combining the above two lemmas we obtain our basic estimate.

\begin{lemma}\label{main_extraction}
Uniformly in $z \geq 1$,
\[ \sum_{\substack{m_1, m_2 \leq z\\ (m_1, m_2) = 1}}
\frac{t(m_1)t(m_2)m_1m_2}{\max(m_1,m_2)^3} \gg \frac{1}{\log z}\left( \sum_{m
\leq z}
\frac{t(m)}{\sqrt{m}} \right)^2
.\]
In particular, for $\log z > 3 \lambda \log \log \lambda$ we have 
\[ \sum_{\substack{m_1, m_2 \leq z\\ (m_1, m_2) = 1}}
\frac{t(m_1)t(m_2)m_1m_2}{\max(m_1,m_2)^3}\geq \exp\left((1 + o(1))
\frac{\lambda}{\log \lambda}\right).\]
\end{lemma}
\begin{proof}
 The left hand side is equal to
\begin{align}\label{main_estimate}\sum_{m_1, m_2 \leq z}
\frac{t(m_1)t(m_2)m_1m_2}{\max(m_1, m_2)^3} 
\sum_{d|(m_1,m_2)}\mu(d)
&= \sum_{\substack{m_1, m_2 \leq z}} \frac{t(m_1)t(m_2)m_1m_2}{\max(m_1,m_2)^3}
\sum_{\substack{d \leq
\frac{z}{\max(m_1,m_2)}\\ (d,m_1m_2)=1}}
\frac{\mu(d)t(d)^2}{d}
\\&  = (1 + o(1)) \sum_{\substack{m_1, m_2 \leq z}}
\frac{t(m_1)t(m_2)m_1m_2}{\max(m_1,m_2)^3} 
\end{align}
by (\ref{mobius_sum}) of Lemma \ref{elementary}.  This last sum is
\[ \gg \sum_{0 \leq k < \log z} \left(\sum_{\frac{z}{e^{k+1}} < m \leq \frac{z
}{e^k}} \frac{t(m)}{\sqrt{m}}\right)^2 \geq \frac{1}{\log z} \left(\sum_{m \leq
z} \frac{t(m)}{\sqrt{m}}\right)^2\] by Cauchy-Schwartz.
\end{proof}

\subsection{Proof of Theorem \ref{large_theorem}}
Recall that we set either $x = \frac{q}{N}$ for $\Delta(N,q)$, or $x =
\frac{1}{2\log q}\sqrt{\frac{q}{N}}$ for $\Delta(\frac{q}{N}, q)$ ($q$
prime) and in either case $\lambda = \sqrt{\log x \log_2 x}$.  Note
that in the case of $\Delta(N,q)$, the condition $N < q
\exp\left(-\frac{2\log q}{\log_2 q}\right)$ guarantees $\lambda^2 >
2\log q$; in particular if $M$ is minimal such that $\sum_{p \leq
  M}\log p > \log q$ as in Proposition \ref{fund_second}, then the
function $r$ is supported on primes greater than $2\log q > M$ so that
$r$ satisfies the conditions of that Proposition.  

Let $z
= \min(N,x)^{\frac{4}{5}}$.  Since we assume $\log N \geq 4\sqrt{\log
  q \log_2 q}\log_3 q$ this guarantees that $z \geq
\exp(3\lambda\log_2 \lambda)$.
By Propositions \ref{fund_second} and \ref{fund_second_dual}, 
both the bound for $\Delta(N,q)$ and $\Delta(\frac{q}{N},q)$ follow
from the estimate
\begin{equation}\label{essential_estimate} \sum_{\substack{n_1,n_2 <
      z\\ (n_1,n_2) = 1}}\frac{r(n_1)r(n_2)n_1n_2}{\max(n_1,n_2)^3}
  \sum_{\substack{g\leq \frac{x}{\max(n_1,n_2)}\\
        (g,n_1n_2)=1}}r(g)^2 \bigg/ \prod_p(1 + r(p)^2)\geq
    \exp\left((2 + o(1))\sqrt{\frac{\log x}{\log_2 x}}\right) .\end{equation}
Applying Rankin's trick to the sum over $g$ with $\alpha =
\frac{1}{(\log \lambda)^3}$ we obtain the desired main term of
\[ \sum_{\substack{m_1,m_2 \leq z\\(m_1,m_2)=1}}
\frac{t(m_1)t(m_2)m_1m_2}{\max(m_1,m_2)^3} \gg \frac{1}{\log
  z}\left(\sum_{m < z} \frac{t(m)}{\sqrt{m}}\right)^2 \geq \exp\left((1 + o(1))
  \frac{\lambda}{\log \lambda}\right)\] with an error term of
\[\left(\prod_{p} (1 + r(p)^2)\right)^{-1}x^{-\alpha}
\sum_{\substack{m_1,m_2 \leq z\\ (m_1,m_2) = 1}}
\frac{r(m_1)r(m_2)(m_1m_2)^{1 + \alpha}}{\max(m_1,m_2)^3}
\sum_{(g,m_1m_2)=1} r(g)^2 g^\alpha.\]  By Lemmas
\ref{ingredients} and \ref{main_extraction}, the ratio of this
error term to the main term is bounded by $\log z$ times the
expression in (\ref{g_tail}), and thus this ratio is $o(1)$.

\bibliographystyle{plain}
\bibliography{resonance}
\end{document}